\documentclass{amsart}

\title{Counting Invariant Curves: a theory of Gopakumar-Vafa
invariants for Calabi-Yau threefolds with an involution}
\author{Jim Bryan and Stephen Pietromonaco}
\date{\today}
\address{
Department of Mathematics\\
University of British Columbia \\
Room 121, 1984 Mathematics Road  \\
Vancouver, B.C., Canada V6T 1Z2  
}

\usepackage{diagbox}
\usepackage{float}
\usepackage{graphicx}
\usepackage{booktabs}
\usepackage{amsmath}
\usepackage{bm}
\usepackage{verbatim}
\usepackage{amsmath,amsthm,amsfonts}
\usepackage{amssymb}
\usepackage{times}
\usepackage{longtable}
\usepackage{caption}
\usepackage{array}
\usepackage{tikz-cd}
\usepackage{mathtools}

\newtheorem{theorem}{Theorem}[section]
\newtheorem{proposition}[theorem]{Proposition}
\newtheorem{conjecture}[theorem]{Conjecture}
\newtheorem{lemma}[theorem]{Lemma}
\newtheorem{corollary}[theorem]{Corollary}

\theoremstyle{definition}

\newtheorem{remark}[theorem]{Remark}
\newtheorem{def-theorem}[theorem]{Definition-Theorem}
\newtheorem{definition}[theorem]{Definition}

\newtheorem{assumption}[theorem]{Assumption}

\usepackage{amsmath}
\usepackage{amsmath,amsthm,amsfonts}
\usepackage{times}

\newcommand{\CC} {{\mathbb C}}          
\newcommand{\NN} {{\mathbb N}}		
\newcommand{\ZZ} {{\mathbb Z}}		
\newcommand{\QQ} {{\mathbb Q}}		
\newcommand{\FF} {{\mathbb F}}
\newcommand{\evir}{e_{\mathsf{vir}}}
\newcommand{\reg}{\mathsf{reg}}
\newcommand{\calE}{\mathcal{E}}

\newcommand{\PP}{\mathbb{P}}
\newcommand{\OO}{\mathcal{O}}
\newcommand{\M}{{\mathsf{M}}}

\newcommand{\nPT}[2]{n^{\scriptscriptstyle{\PT} }_{#1}\left(#2\right)}
\newcommand{\inv}{\imath}

\newcommand{\nPTbg}{\nPT{g}{\beta }}
\newcommand{\nPTbgh}{\nPT{g,h}{\beta }}

\newcommand{\nMT}[2]{n^{\scriptscriptstyle{\MT} }_{#1}\left(#2 \right)}
\newcommand{\nMTbg}{\nMT{g}{\beta }}
\newcommand{\nMTbgh}{\nMT{g,h}{\beta }}

\newcommand{\nbg}{n_{g}(\beta )}
\newcommand{\nbgh}{n_{g,h}(\beta )}
\newcommand{\Sresolve}{\widehat{S}}
\newcommand{\Ehat}{\widehat{E}}

\newcommand{\PT}{\mathsf{PT}}
\newcommand{\MT}{\mathsf{MT}}

\newcommand{\Chow}{ \operatorname{Chow}}
\newcommand{\supp}{ \operatorname{supp}}

\newcommand{\exc}{\mathsf{exc}}
\newcommand{\pt}{\mathsf{pt}}
\newcommand{\sh}{\mathsf{sh}}
\newcommand{\type}{\mathsf{type}}
\newcommand{\evNiI}{{\mathsf{N_{I}^{ev}}}}
\newcommand{\oddNiI}{{\mathsf{N_{I}^{odd}}}}
\newcommand{\NiII}{{\mathsf{N_{II}}}}
\newcommand{\evAb}{{\mathsf{A^{ev}}}}
\newcommand{\oddAb}{{\mathsf{A^{odd}}}}

\newcommand{\dsemistable}{\operatorname{\delta-ss}}
\newcommand{\dstable}{\operatorname{\delta-s}}

\newcommand{\Jac}{\operatorname{J\overline{ac}}}

\newcommand{\half}{\frac{1}{2}}
\newcommand{\tinyhalf}{\tfrac{1}{2}}

\newcommand{\coker}{\operatorname{coker}}

\newcommand{\Sym}{\operatorname{Sym}}

\newcommand{\Pic}{\operatorname{Pic}}

\newcommand{\Tot}{\operatorname{Tot}}

\renewcommand\restriction[2]{{
  \left.\kern-\nulldelimiterspace 
  #1 
  \vphantom{\big|} 
  \right|_{#2} 
  }}

\begin{document}

\begin{abstract}
We develop a theory of Gopakumar-Vafa (GV) invariants for a Calabi-Yau
threefold (CY3) $X$ which is equipped with an involution $\inv$
preserving the holomorphic volume form. We define integers
$\nbgh $ which give a virtual count of the number of genus $g$
curves $C$ on $X$, in the class $\beta \in H_{2}(X)$, which are
invariant under $\inv$, and whose quotient $C/\inv$ has genus $h$.  We
give two definitions of $\nbgh $ which we conjecture to be
equivalent: one in terms of a version of Pandharipande-Thomas theory
and one in terms of a version of Maulik-Toda theory.

We compute our invariants and give evidence for our conjecture in
several cases. In particular, we compute our invariants when
$X=S\times \CC$ where $S$ is an Abelian surface with $\inv (a)=-a$ or
a $K3$ surface with a symplectic involution (a Nikulin $K3$
surface). For these cases, we give formulas for our invariants in
terms of Jacobi modular forms. For the Abelian surface case, the
specialization of our invariants $\nbgh $ to $h=0$ recovers the count
of hyperelliptic curves on an Abelian surface first computed in
\cite{BOPY}.
\end{abstract}

\maketitle 

\markboth{Counting Invariant Curves}  {Bryan and Pietromonaco}


\section{Ordinary GV invariants}\label{sec: Ordinary GV invariants}
Let $X$ be a Calabi-Yau threefold (CY3) by which we mean a smooth
quasi-projective variety over $\CC$ of dimension three with
$K_{X}\cong \OO_{X}$. In 1998 \cite{Go-Va}, Gopakumar and Vafa (GV)
defined via physics integer invariants $\nbg $ which give a
virtual count of curves $C\subset X$ of genus $g$ and class $[C]\in
H_{2}(X)$.

Mathematically, there are two conjecturally equivalent sheaf theoretic
approaches to defining $\nbg $, one by Pandharipande-Thomas
(PT) via their stable pair invariants \cite{Pandharipande-Thomas3},
and one more recently given by Maulik and Toda (MT) using perverse
sheaves \cite{Maulik-Toda}. We begin by reviewing ordinary GV theory,
and then we develop in a parallel fashion a theory of GV invariants
for CY3s with an involution.

\subsection{GV invariants via PT theory}
Let $\PT_{\beta ,n}(X)$ be the moduli space of PT pairs
\cite{Pandharipande-Thomas1}:
\begin{equation*}
\PT_{\beta ,n}(X) = \{\,\,(F,s): \,\,  s\in H^{0}(X,F), \,\,
[\supp(F)]=\beta , \,\, \chi (F)=n\}
\end{equation*}
where $F$ is a coherent sheaf on $X$ with proper support of pure
dimension 1, and $\coker (s)$ has support of dimension 0.

For any scheme $S$ over $\CC$, Behrend \cite{Behrend-micro} defined a
constructible function $\nu_{S}:S\to \ZZ$ and we define the
\emph{virtual Euler characteristic} to be the Behrend function
weighted topological Euler characteristic:

\[
\evir (S) = e(S,\nu_{S}) = \sum_{k\in \ZZ} k\cdot e\left(\nu_{S}^{-1}(k) \right).
\]

We then define the \emph{PT invariants} by 
\[
N^{\PT}_{\beta ,n}(X) = \evir (\PT_{\beta ,n}(X))
\]
and the \emph{PT partition function} by
\[
Z^{\PT}(X) = \sum_{\beta,n} N^{\PT}_{\beta ,n}(X)\,  Q^{\beta} \, y^{n}  .
\]

\begin{definition}\label{defn: GV via PT}
The Gopakumar-Vafa invariants (via PT theory) $\nPTbg$ are
defined via the following equation:
\begin{equation}\label{eqn: log(ZPT)=sum 1/k etc}
\log Z^{\PT}(X) = \sum_{k>0} \sum_{\beta,g}\, \frac{1}{k}\cdot Q^{k\beta}\cdot 
\nPTbg \cdot  \psi^{g-1}_{-(-y)^{k}}
\end{equation}
where 
\[
\psi_{x} =2+x+x^{-1}
\]
\end{definition}

\begin{remark}
Writing $\log Z^{\PT}(X)$ in the form given by the righthand side of
Equation~\eqref{eqn: log(ZPT)=sum 1/k etc} uses the fact that the
coefficient of $Q^{\beta}$ in $Z^{\PT}(X)$ is the Laurant expansion of a
rational function in $y$ which is invariant under $y \leftrightarrow
y^{-1}$ \cite{Pandharipande-Thomas1,Bridgeland-PTDT,
Toda-PTDT}. Although it isn't clear from the definition, one expects
that $\nPTbg =0$ if $g<0$.
\end{remark}

\begin{remark}
Gopakumar and Vafa gave a formula relating their invariants to the
Gromov-Witten (GW) invariants. Equation~\eqref{eqn: log(ZPT)=sum 1/k etc} is
equivalent to the Gopakumar-Vafa formula after using the expected
relationship between PT and GW invariants. 
\end{remark}

\bigskip

\subsection{GV invariants via MT theory}
We define the moduli space of  Maulik-Toda (MT) sheaves to be
\[
\M_{\beta}(X) = \{F : [\supp(F)]=\beta ,\,  \chi (F)=1 \}
\]   
where $F$ is a coherent sheaf on $X$ with proper sheaf theoretic
support of pure dimension 1 and where $F$ is \emph{Simpson stable},
which in this case is equivalent to the condition that if $F'
\subsetneq F$, then $\chi (F')\leq 0$.

The MT moduli space is a quasi-projective scheme and it has  a proper
morphism to the Chow variety given by the Hilbert-Chow morphism:
\begin{align*}
\pi : \M_{\beta}(X) &\to \Chow_{\beta}(X)\\
[F] &\mapsto \supp (F)
\end{align*}

There is a perverse sheaf $\phi^{\bullet}$ on $ \M_{\beta}(X)$ which is
locally given by the perverse sheaf of vanishing cycles associated to
the local super-potential (the moduli space is locally the critical
locus of a holomorphic function on a smooth space, the so-called
super-potential). The construction of $\phi^{\bullet}$ was done in
\cite{BBDJS}, and requires the choice of ``orientation data'' : a
squareroot of the virtual canonical line bundle on
$\M_{\beta}(X)$. Maulik and Toda conjecture the existence of a
canonical choice of orientation data (one that is compatible with the
morphism $\pi$). Using that choice, the \emph{Maulik-Toda} polynomial
is defined as follows:
\[
\MT_{\beta}(y) = \sum_{i\in \ZZ}\,  \chi (\,{}^{p} \hspace{-1pt} H^{i}
(R^{\bullet}\pi_{*} \phi^{\bullet})) y^{i}
\]
where ${}^{p} \hspace{-1pt} H^{i} (-)$ is the $i$th cohomology functor
with respect to the perverse $t$-structure. By self-duality of
$\phi^{\bullet}$ and Verdier duality, $\MT_{\beta}(y)$ is an integer
coefficient Laurent polynomial in $y$ which is invariant under
$y\leftrightarrow y^{-1}$. Noting that $\{\psi_{y}^{g} \}_{g\geq 0}$
forms an integral basis for such polynomials, we may write the MT
polynomial as follows:

\begin{definition}
The GV invariants (via MT theory) $\nMTbg $ are
defined by the equation
\[
\MT_{\beta}(y) = \sum_{g\geq 0} \nMTbg  \,\psi_{y}^{g}. 
\]
\end{definition}

The main conjecture of Maulik and Toda is
\begin{conjecture}\label{conj: nPTbg = nDTbg}
$\nMTbg  = \nPTbg$.
\end{conjecture}

\begin{remark}
Compared to the definition via PT theory, the above definition of GV
invariants is more directly tied to the geometry of curves in the
class $\beta$ and more closely matches the original physics
definition. In particular, the invariants $\nMTbg $ only involve the
single moduli space $\M_{\beta}(X)$. In contrast, the invariants
$\nPTbg $ involve a subtle combination of the PT invariants associated
to an infinite number of moduli spaces, namely the spaces $\PT_{\beta
',n}(X)$ where $\beta =k\beta '$ and $n$ is unbounded from above. From
this perspective, the conjectural formula $\nMTbg = \nPTbg $ can be
viewed as a kind of multiple cover formula for PT invariants.
\end{remark}

\subsection{The local $K3$ surface: the Katz-Klemm-Vafa (KKV) formula.}\label{subsec: KKV formula}

Suppose that $S$ is a $K3$ surface and $\beta_{d}\in H_{2}(S)$ is a
curve class with $\beta_{d}^{2}=2d$. The CY3 $X=S\times \CC $ is
sometimes called the local $K3$ surface. In this case,
Conjecture~\ref{conj: nPTbg = nDTbg} holds and $\nMT{g}{\beta_{d}} =
\nPT{g}{\beta_{d}}$ only depends on $d$ and $g$ (and not the
divisibility of $\beta_{d}$). The invariants $\nPT{g}{\beta_{d}}$ and
$\nMT{g}{\beta_{d}}$ were computed (in full generality) by
Pandharipande and Thomas \cite{Pandharipande-Thomas-KKV} and Shen and
Yin \cite[Thm~0.5]{Shen-Yin-hyperkahler} respectively.  The MT
polynomials are given by the famous KKV formula first conjectured by
Katz, Klemm, and Vafa \cite{KKV}:
\begin{equation}\label{eqn: the KKV equation}
\sum_{d=-1}^{\infty} \MT_{\beta_{d}}(y) q^{d} = - 
q^{-1}\prod_{n=1}^{\infty} (1-q^{n})^{-20} (1+yq^{n})^{-2}(1+yq^{n})^{-2} 
\end{equation}
The right hand side can also be written as $\psi_{y}\cdot
\phi_{10,1}(q,-y)^{-1}$ where $\phi_{10,1}(q,y)$ is the Fourier expansion
of the unique Jacobi cusp form of weight $10$ and index $1$.

The fact that $n_{g}(\beta_{d})$ is independent of the divisibility
of $\beta_{d}$ is a deep and unusual feature of the local $K3$
geometry.

\bigskip

\section{GV invariants for CY3s with an involution}

Let $X$ be a CY3 equipped with an involution 
\[
\inv :X\to X
\]
such that the induced action on $K_{X}$ is trivial.  The purpose of
this article is to develop a theory of GV invariants which count
$\inv$-invariant curves on $X$. Namely, we seek to define integers
$\nbgh $ which give a virtual count of genus $g$,
$\inv$-invariant curves $C\subset X$ with $[C]=\beta \in H_{2}(X)$,
and such that the genus of the quotient $C/\inv$ is $h$.

We develop this theory parallel to the presentation of the ordinary GV
invariants given in Section~\ref{sec: Ordinary GV invariants}. Namely
we define invariants $\nPTbgh$ and $\nMTbgh$ in terms of a version of
PT and MT theory respectively and we conjecture that they are
equal.

We do not currently know how to include curves that are \emph{fixed}
by $\inv$ (as opposed to merely invariant) and so we make the
following general assumption:

\begin{assumption}\label{assump: no fixed curves in beta}
Throughout, we assume that the curve classes $\beta \in H_{2}(X)$
that we consider do not admit an effective decomposition $\beta
=\sum_{i} d_{i}C_{i}$ containing an $\inv$-fixed curve $C_{i}$. 
\end{assumption}

\subsection{ $\inv$-GV invariants via PT theory.}    
We denote by $R_{+}$ and $R_{-}$ the trivial and the nontrivial
irreducible representation  of the group of order 2 and we let
$R_{\reg}=R_{+}\oplus R_{-}$ denote the regular representation.

Recall that a sheaf $F$ on $X$ is \emph{ $\inv$-invariant} if
$\inv^{*}F\cong F$ and that an \emph{$\inv$-equivariant} sheaf is an
$\inv$-invariant sheaf $F$ along with a choice of a lift of $\inv$ to
an isomorphism $\tilde{\inv}:\inv^{*}F\to F$.

If $F$ is an $\inv$-equivariant sheaf then
\[
\chi (F) = \sum_{k} (-1)^{k} H^{k}(X,F)
\]
is naturally a virtual representation and so can be written in the
form
\[
\chi (F) = nR_{\reg}+\epsilon R_{-}
\]
for some $n,\epsilon \in \ZZ$.

We define the space of $\inv$-equivariant PT pairs
\[
\PT_{\beta ,n,\epsilon }(X,\inv ) = \{\,\,(F,s): \,\,  s\in H^{0}(X,F), \,\,
[\supp(F)]=\beta , \,\, \chi (F)=nR_{\reg}+\epsilon R_{-}\}
\]
where $(F,s)$ is a PT pair such that $F$ is an $\inv$-equivariant sheaf
and $s$ is an equivariant section. We note that

\begin{equation}\label{eqn: decomp of inv PT moduli space}
\PT_{\beta ,d}(X)^{\inv} = \bigsqcup_{2n+\epsilon =d}\,\, \PT_{\beta
,n,\epsilon}(X,\inv)
\end{equation}
since the points of the $\inv$-fixed locus $\PT_{\beta ,n}(X)^{\inv}$
corresponds to $\inv$-invariant PT pairs, but each $\inv$-invariant PT
pair has a unique $\inv$-equivariant structure making the section
equivariant.

\begin{definition}
We define the $\inv$-PT invariants and the $\inv$-PT partition function by
\begin{align*}
N^{\PT}_{\beta ,n,\epsilon}(X,\inv ) &= \evir (\PT_{\beta ,n,\epsilon}(X,\inv ))\\
Z^{\PT}(X,\inv ) &= \sum_{\beta,n,\epsilon}   N^{\PT}_{\beta
,n,\epsilon}(X,\inv )\,  Q^{\beta}\,  y^{n}\,  w^{\epsilon} 
\end{align*}
where in the sum, $\beta$ is summed over the semi-group of classes
satisfying Assumption~\ref{assump: no fixed curves in beta}. 
\end{definition}

These invariants are not new: they can be recovered from the orbifold
PT invariants of the stack quotient $[X/\inv ]$ studied for example in
\cite{Beentjes-Calabrese-Rennemo}. It follows from
\cite[Prop~7.19]{Beentjes-Calabrese-Rennemo}, that the coefficient of
$Q^{\beta}$ in $Z^{\PT} (X,\inv )$ is a Laurent expansion of a
rational function in $y$ and $w$ which is invariant under
$y\leftrightarrow y^{-1}$ and $w\leftrightarrow w^{-1}$. This allows
us to make the following definition:
\begin{definition}\label{defn: nptbgh}
The $\inv$-GV invariants (via PT theory) $\nPTbgh$ (for classes
$\beta$ satisfying Assumption~\ref{assump: no fixed curves in beta})
are defined by the formula
\[
\log (Z^{\PT}(X,\inv)) = \sum_{k>0}\sum_{\beta ,g,h}
\frac{1}{k}Q^{k\beta} \cdot \nPTbgh \cdot
\psi^{h-1}_{-(-y)^{k}}\cdot \psi_{w^{k}}^{g+1-2h} 
\]
where as before $\psi_{x} = 2+x+x^{-1}$. 
\end{definition}

\begin{remark}
The number $g+1-2h$ is half the number of fixed points on a smooth
genus $g$ curve with an involution having a quotient of genus $h$.
\end{remark}

\begin{remark}
Although it is not apparent from this definition, we expect $\nPTbgh$
to have good finiteness properties, namely that for fixed $\beta$ we
expect $\nPTbgh$ to be non-zero for only a finite number of values of
$(g,h)$ and that those values should have $h\geq 0$ and $g\geq
-1$. The possible occurence of non-zero counts for $g=-1$ is due to
(for example) the possibility of invariant curves $C=C_{1}\cup C_{2}$
consisting of a disjoint $\inv$-orbit of rational curves. Such a curve
should be interpreted as having genus $-1$ via the formula $\chi
(\OO_{C})=1-g$.
\end{remark}
     
\subsection{$\inv$-GV invariants via MT theory.} We define the
moduli space of $\inv$-MT sheaves to be
\begin{equation}\label{equation:MT moduli space definition}
\M^{\epsilon}_{\beta}(X,\inv ) = \{F:\,\, [\supp (F)]=\beta , \,\,
\chi (F)=R_{\reg}+\epsilon R_{-} \} 
\end{equation}
where $F$ is an $\inv$-equivariant coherent sheaf on $X$ with proper
sheaf theoretic support of pure dimension one and where $F$ is
\emph{$\inv$-stable}:

\begin{definition}\label{defn: tau-stable MT sheaves}
We say an $\inv$-equivariant sheaf $F$ on $X$ of pure dimension one with
$\chi (F)=R_{\reg}+\epsilon R_{-}$ is \emph{$\inv$-stable} if all
$\inv$-equivariant subsheaves $F'\subsetneq F$, with $\chi
(F')=kR_{\reg} +\gamma  R_{-}$ satisfy $k\leq 1$ and if $k=1$, then
$\gamma  <\epsilon$ and $[\supp (F')]=[\supp (F)]\in H_{2}(X)$.
\end{definition}

\begin{remark}
We will show in Proposition~\ref{prop: Nironi stability is equivalent to
i-stability} that $\inv$-stability can be reformulated in terms of
Nironi stability for the corresponding sheaf on the stack quotient
$[X/\inv ]$. A consequence is that $\M^{\epsilon}_{\beta}(X,\inv )$
is a scheme and it is proper over $\Chow_{\beta}(X)$.
\end{remark}

Let 
\begin{equation}\label{equation:Hilbert-Chow morphism}
\pi^{\epsilon} : \M^{\epsilon}_{\beta}(X,\inv )\to \Chow_{\beta}(X)
\end{equation}
be the Hilbert-Chow morphism. Since $\M^{\epsilon}_{\beta}(X,\inv )$
parameterizes objects in the CY3 category of $\inv$-equivariant
coherent sheaves on $X$, there exists a perverse sheaf of vanishing
cycles\footnote{As in \cite[Defn 2.7]{Maulik-Toda}, we assume that our
orientation is \emph{strictly Calabi-Yau}.} $\phi^{\bullet}$ on
$\M^{\epsilon}_{\beta}(X,\inv )$ and we can define the $\inv$-MT
polynomial in a fashion analogous to the ordinary MT polynomial:
\begin{equation}\label{equation:tau-MT polynomial}
\MT_{\beta}(y,w) =  \sum_{i,\epsilon \in \ZZ}\,  \chi (\,{}^{p} \hspace{-1pt} H^{i}
(R^{\bullet}\pi^{\epsilon }_{*} \phi^{\bullet})) \, y^{i} \, w^{\epsilon}.
\end{equation}
As before, self-duality and Verdier duality imply that
$\MT_{\beta}(y,w)$ is a Laurent polynomial in $y$ invariant under
$y\leftrightarrow y^{-1}$. We conjecture that in general
$\MT_{\beta}(y,w)$ is also a Laurent invariant polynomial in $w$
invariant under $w\leftrightarrow w^{-1}$. Assuming this conjecture,
we can write $\MT_{\beta}(y,w)$ as a polynomial in $\psi_{y}$ and
$\psi_{w}$ and make the following definition.
\begin{definition}\label{defn: tau-GV via MT}
The $\inv$-GV invariants (via MT theory) $\nMTbgh$ (for classes
satisfying Assumption~\ref{assump: no fixed curves in beta}) are
defined by the formula:
\[
\MT_{\beta}(y,w) = \sum_{g,h} \nMTbgh \,
\psi_{y}^{h}\,\psi_{w}^{g+1-2h}. 
\]
\end{definition}

Our main conjecture is that our two definitions of $\inv$-GV
invariants are equivalent.
\begin{conjecture}\label{conj: two tau-GV defns are the same nPT=nMT}
$\nPTbgh =\nMTbgh$. 
\end{conjecture}

\bigskip

\subsection{Examples: local Abelian surfaces and local Nikulin $K3$
surfaces.  } One of the main results of this paper are various
$\inv$-equivariant versions of the KKV formula. Namely, we compute our
invariants and prove our conjecture for the case $X=S \times \CC $
where $S$ is either an Abelian or $K3$ surface and where $\inv$ acts
trivially on $\CC $ and symplectically on $S$.

For the case of an Abelian surface, the involution is the natural one
arising from the group structure: $\inv (a) = -a$. A $K3$ surface
equipped with a symplectic involution is called a Nikulin surface and
there are two distinct deformation types which we call Type (I) and
Type (II) (see Definition~\ref{defn: Type I and Type II}).

As with the ordinary GV invariants of a local $K3$ surface, our
invariants admit a surprising lack of dependency on the divisibility
of the curve class and they are given by formulas involving Jacobi
modular forms:

\begin{theorem}\label{thm: Formulas for nPTgh for X=SxC}
Let $X=S\times \CC$ where $S$ is an Abelian surface or a Nikulin $K3$
surface, and let $\beta\in H_{2}(S)$ be an effective invariant curve
class with $\beta^{2}=2d$. Then:
\begin{enumerate}
\item if $S$ is a Type (II) Nikulin surface, $\nPT{g,h}{\beta }$ only
depends on $(g,h,d)$.  In particular, it doesn't depend on the
divisibility of $\beta$.  \vskip2ex
\item if $S$ is an Abelian surface or a Type (I) Nikulin surface,
$\nPT{g,h}{\beta }$ only depends on $(g,h,d)$ as well as the parity of
the divisibility of $\beta$.
\end{enumerate}
We denote these invariants by $n_{g,h}(d;\type)$ where $\type \in
\{\evAb ,\oddAb ,\evNiI ,\oddNiI ,\NiII \}$ distinguishes the cases in
the obvious way.  Then the invariants are determined from the formula:
\[
\sum_{g,h} n_{g,h}(d; \type) \psi_{y}^{h-1}\psi_{w}^{g+1-2h} =
\left[\frac{\Theta_{T}(q^{2},w)}{\phi_{10,1}(q^{2},-y)}
\right]_{q^{d}}
\]
where $[\dotsb ]_{q^{d}}$ denotes the coefficient of $q^{d}$ in the
expression $[\dotsb ]$.

Moreover, $T$ is a lattice or shifted lattice depending on the type,
$\Theta_{T}(q^{2}, w)$ is an explicitly determined Jacobi theta
function (see Theorem \ref{thm: main statement about nPT for
Abelian/Nikulin surfaces}), and $\phi_{10,1}(q,y)$ is the unique
Jacobi cusp form of weight $10$ and index $1$.  In particular, for
types $\oddAb$ and $\oddNiI$ we get infinite product formulas:
\begin{align*}
\sum_{g,h,d}
n_{g,h}(d;\oddAb) \,\psi_{y}^{h}\,\psi_{w}^{g-1-2h}\, q^{d} &=
-4\prod_{n=1}^{\infty}
\frac{(1+q^{n})^{8}(1+wq^{n})^{4}(1+w^{-1}q^{n})^{4}}{(1-q^{2n})^{4}
(1+yq^{2n})^{2} (1+y^{-1} q^{2n})^{2}}\,\, ,\\
 \sum_{g,h,d} n_{g,h}(d;\oddNiI) \,\psi_{y}^{h}\,\psi_{w}^{g-2h}\,
q^{d+1} &=\,\,\,-
\prod_{n=1}^{\infty}\frac{(1+q^{n})^{4}(1+wq^{n})^{2}(1+w^{-1}q^{n})^{2}}{(1-q^{2n})^{12}(1+yq^{2n})^{2}(1+y^{-1}q^{2n})^{2}}\,\,.
\end{align*}
\end{theorem}

\begin{remark}
The specialization of the invariants $n_{g,h}(\beta )$ to $h=0$ count
$\inv$-invariant hyperelliptic curves. The problem of counting the
number of genus $g$ hyperelliptic curves in a primitive class
$\beta_{d}$ on an Abelian surface $A$ was first considered by Rose
\cite{Rose-2014} and then solved by Bryan-Oberdieck-Pandharipande-Yin
\cite{BOPY}. We may specialize our invariants $\nPT{g,h}{d;\oddAb}$ to
$h=0$ by setting $y=-1$. The above formula then becomes
\[
\sum_{d=0}^{\infty} \sum_{g>0} n_{g,0}(d;\oddAb) \psi_{w}^{g-1}\, q^{d} =
-4\prod_{n=1}^{\infty} \frac{(1+wq^{n})^{4}(1+w^{-1}q^{n})^{4}}{(1-q^{n})^{8}}
\]
We note that the invariant $n_{g,0}(d;\oddAb)$ is equal to
$\mathsf{h^{Hilb}_{g,A,\beta_{d}}}$ in the notation of \cite{BOPY} and
the above formula is equivalent to the equation in Proposition 4 of
\cite{BOPY}. 
\end{remark}
\begin{remark}
For the case where $\beta_{d}$ is the primitive class on an Abelian
surface, our invariants $n_{g,h}(d;\oddAb )$ are refinements of the
invariants $\mathsf{{n}}_{d}(h)$ considered in
\cite{Pietromonaco-GHilbA}. The relationship is given by
\[
\mathsf{n}_{d}(h) = -\sum_{g} 4^{g-2h}\cdot  n_{g,h}(d, \oddAb).
\]
\end{remark}

Our main technique to prove Theorem~\ref{thm: Formulas for nPTgh for
X=SxC} / Theorem~\ref{thm: main statement about nPT for
Abelian/Nikulin surfaces} is to use the Donaldson-Thomas Crepant
Resolution Conjecture (DT-CRC)
\cite{Bryan-Cadman-Young,Beentjes-Calabrese-Rennemo} to compute
orbifold PT invariants in terms of the crepant resolution which in
this case is a local $K3$ surface. We can then apply the KKV formula
making crucial use of the description of the Picard lattice of Kummer
$K3$ surfaces and Nikulin resolutions given by Garbagnati-Sarti
\cite{Garbagnati-Sarti-Kummer-surfaces-2016,
Garbagnati-Sarti-even-set-2008}. This is carried out in
Section~\ref{sec: Local Abelian/Nikulin Surfaces (PT theory)}.

In Section~\ref{sec: Local Abelian and Nikulin Surfaces (MT theory)},
we use the derived McKay correspondence to compute the MT versions of
our invariants $\nMT{g,h}{\beta} $ for $X=S\times \CC$. The final
result is

\begin{theorem}\label{thm: nPT=nMT for SxC}
The two definitions of $\inv$-GV
invariants coincide for $X=S\times \CC$ for $S$ an Abelian surface or
a Nikulin $K3$ surface:
\[
\nPTbgh =\nMTbgh .
\]
\end{theorem}

\bigskip

\subsection{Further examples} In Section~\ref{sec: additional
examples}, we give two other examples illustrating our theory and
providing evidence of our conjecture. For the case where $X\to B$ is
an elliptically fibered CY3 over a surface $B$ with integral fibers
and $\inv :X\to X$ is the composition of an involution on $B$ and
fiberwise multiplication by $-1$, we compute both $\nPT{g,h}{\beta}$
and $\nMT{g,h}{\beta}$ completely for $\beta =[F]$, the class of the
fiber. Let $C\subset B$ be the fixed locus of the involution on the
base and let $S=\restriction{X}{C}$ be the restriction of $X\to B$ to
$C$. The result is the following:
\begin{theorem}
Let $(X,\inv )$ be an elliptically fibered CY3 with notation as
above. Then for all $g$ and $h$, $\nPT{g,h}{[F]}=\nMT{g,h}{[F]}$ and they
are given by
\[
n_{g,h}([F]) = \begin{cases} -e(C) &\text{if $(g,h)=(1,0)$,}\\
e\left(S \right)&\text{if $(g,h)=(0,0)$,}\\
0&\text{ otherwise.}
\end{cases}
\]
\end{theorem}

We also consider the case where $X=\operatorname{Tot}(\OO (-1)\oplus
\OO (-1))$ is the conifold, $\inv :X\to X$ is the Calabi-Yau
involution which acts non-trivially on the base, and $C\subset X$ is
the zero section. We use the orbifold
topological vertex to compute $\nPT{g,h}{d[C]}$ and we use stability
considerations to compute $\nMT{g,h}{d[C]}$. The result is
\begin{proposition}\label{prop: nPT=nMT=1 for (g,h,d)=1 for local
football}
For $X$ the conifold with $\inv$ as above,
\[
\nPT{g,h}{d[C]} =\nMT{g,h}{d[C]}= \begin{cases}
1&\text{ if $(g,h,d)=(0,0,1)$,}\\
0&\text{ otherwise.}
\end{cases}
\]
\end{proposition}
We remark that despite the simplicity of the answer, the above
proposition is the result of rather involved orbifold topological
vertex computation which was done in \cite{Bryan-Cadman-Young} and it
gives a non-trivial instance of our conjecture.

\section{Motivating example of an isolated smooth invariant curve.}

The simplest evidence of our conjecture is the case of a rigid local
curve in the primitive class of the curve.  Let $C$ be a non-singular
curve of genus $g$.  Suppose there exists an involution $\inv : C \to
C$ with fixed points
\[
C^{\inv} = \{p_{1}, \ldots, p_{2m}\}.
\] 
If $h$ is the genus of the quotient $C/\inv$, then by Riemann-Hurwitz we have
\[
m = g+1-2h.
\]

Let $N$ be an $\inv$-equivariant line bundle on $C$ such that
$H^{0}(N)=H^{1}(N)=0$. Then 
\[
X=\Tot (N\oplus K_{C}N^{-1})
\]
is a CY3 with an induced involution (also denoted $\inv$) acting
trivially on $K_{X}$. Let $[C]$ be the class of the zero section
$C\hookrightarrow X$ which we note is rigid.

The moduli spaces for the class $[C]$ can be determined explicitly:

\begin{proposition}\label{prop: iPT and iMT moduli spaces for degree 1
local curve}
The $\inv$-PT and $\inv$-MT moduli spaces in the class $[C]$ are given by
\begin{align*}
\PT_{[C],n,\epsilon}(X, \inv)& = \coprod_{\substack{ T \subseteq \{1,\dotsc ,2m \} \\ |T| = \epsilon + m}} \Sym^{n+h-1} \big(C/\inv \big), \\
\M_{[C]}^{\epsilon}(X, \inv)& = \coprod_{\substack{ T \subseteq \{1,\dotsc ,2m \} \\ |T| = \epsilon + m}} \Pic^{h} \big(C/\inv \big) .
\end{align*}
\end{proposition}

As a consequence, we find there is a single non-zero $\inv$-GV
invariant in the class $[C]$:

\begin{corollary}\label{cor: ngh(C)=1 for local curve}
\[
\nPT{g',h'}{[C]} = \nMT{g',h'}{[C]} = \begin{cases}
1\text{ if $(g',h')=(g,h),$}\\
0\text{ if $(g',h')\neq (g,h).$}
\end{cases}
\]
\end{corollary}

\begin{remark}\label{rem: multiple covers of local curve and orbifold P=W}
We expect the invariants $n_{g',h'}(d[C])$ to be complicated for $d>1$
and $g>0$. For the related case of $X=\Tot (K_{C}\oplus \OO )$,
Conjecture~\ref{conj: nPTbg = nDTbg} is equivalent to the well known
$P=W$ conjecture for the $\operatorname{GL}_{d}$ Hitchin system on $C$
\cite[\S~9.3]{Maulik-Toda}. We expect that for $\inv : X\to X$, the
natural lift of an involution on $C$, our Conjecture~\ref{conj: two
tau-GV defns are the same nPT=nMT} should be equivalent to an orbifold
version of $P=W$ for the orbifold curve $[C/\inv ]$.
\end{remark}

To prove Proposition \ref{prop: iPT and iMT moduli spaces for degree 1
local curve}  and Corollary~\ref{cor: ngh(C)=1 for local curve}, we
need the following lemma.

\begin{lemma}\label{lem: chi(L) for L=O(D) invariant under i}
Let $L=\OO (D)$ be an $\inv$-equivariant line bundle on $C$ admitting
an $\inv$-invariant section $\OO_{C}\to L$ vanishing on an
$\inv$-invariant effective divisor $D$. Then $D$ can be written
\[
D= \sum_{j} d_{j}(x_{j}+\inv (x_{j})) +\sum_{i\in T} p_{i}
\]
where $T\subset \{1\dotsb 2m \}$, $x_{j}\in C$,  and 
\[
\chi (L) = nR_{\reg}+\epsilon R_{-}
\]
with
\[
n=1-h+\sum_{j}d_{j} , \quad \epsilon =|T|-m.
\]
Moreover, the above formula still holds for $L=\OO (D)$ with $D$
$\inv$-invariant, but not necessarily effective. 
\end{lemma}
\begin{proof}
Since the support of $D$ is $\inv$-invariant, it must consist of free
orbits and fixed points. Absorbing multiplicities which are not zero
or one on the $p_{i}$'s into the first term using $2p_{i}=p_{i}+\inv
(p_{i})$ we see $D$ may be written in the given form. Assuming $D$ is
effective, we then get
sequence
\[
0\to \OO_{C}\to L\to \OO_{D}(D) \to 0
\]
and consequently we find
\[
\chi (L) = \chi (\OO_{C}) +(\sum_{j}d_{j})R_{\reg} +\sum_{i\in T} \chi (\OO_{p_{i}}(p_{i})).
\]
Using a local coordinate about $p_{i}$, it is easy to determine that
$\chi (\OO_{p_{i}}(p_{i}))=R_{-}$. We also observe that 
\[
\chi (\OO_{C}) = (1-h)R_{+} + (h-g)R_{-} = (1-h)R_{\reg} - mR_{-}
\]
and the formula for $\chi (L)$ follows. The general case is then
obtained by writing $D=D' - D''$ with $D'$ and $D''$ effective and
using the sequence
\[
0\to \OO(D'-D'')\to \OO(D')\to \OO_{D''}(D')\to 0.
\]
\end{proof}

\begin{proof}[Proof of Proposition~\ref{prop: iPT and iMT moduli
spaces for degree 1 local curve}] An $\inv$-PT pair on $X$ in the
class $[C]$ must be supported on $C$ and hence be an
$\inv$-equivariant line bundle with a non-zero invariant section
$\OO_{C}\to L$. Such $\inv$-PT pairs are determined up to isomorphism
by the associated invariant divisor $D$ which by Lemma~\ref{lem:
chi(L) for L=O(D) invariant under i} is determined by the subset
$T\subset \{1,\dotsc ,2m \}$ and $\sum_{j} d_{j}=n+h-1$ $\inv$-orbits
on $C$ or equivalently, $n+h-1$ points on $C/\inv$. The first equation
in the proposition follows.

For the same reasons, an $\inv$-MT sheaf on $X$ in the class $[C]$
must be an $\inv$-equivariant line bundle $L\to C$, or equivalently, a line
bundle on the stack quotient $[C/\inv ]$. Since $\chi (L) =
R_{\reg}+\epsilon R_{-}$, $L$ admits a non-zero $\inv$-invariant
section and hence is of the form $L\cong \OO_{C}(D)$ where $D$ is an
invariant divisor given as in Lemma~\ref{lem: chi(L) for L=O(D)
invariant under i} with $\sum_{j}d_{j}=h$.

The Picard group of an orbifold is given in general in 
\cite[\S~B]{Popa2017}. In particular for $[C/\inv ]$ we have
\begin{align*}
\Pic ([C/\inv ])&\cong \Pic^{0}(C/\inv )\oplus H^{2}_{\mathsf{orb}}([C/\inv ])\\
&\cong \Pic^{0}(C/\inv )\oplus \ZZ \oplus (\ZZ /2)^{2m}\\
&\cong \Pic (C/\inv ) \oplus (\ZZ /2)^{2m}.
\end{align*}
Under the above isomorphism, the line bundle $\OO_{C}(D)$ with $D$ as
in the lemma goes to
$(\OO_{C/\inv}(\sum_{j}d_{j}\overline{x}_{j}),1_{T})$ where
$\overline{x}_{j}\in C/\inv $ is the point corresponding to the orbit
$x_{j}+\inv (x_{j})$ and $1_{T}=(t_{1},\dotsc ,t_{2m})$ where
$t_{i}=1$ if $i\in T$ and $0$ otherwise. The second equation of
Proposition~\ref{prop: iPT and iMT moduli spaces for degree 1 local
curve} follows.
\end{proof}

\begin{proof}[Proof of Corollary~\ref{cor: ngh(C)=1 for local curve}]
Since $[C]$ is a primitive class, we have 
\begin{align*}
\left[ Z^{\PT}(X,\inv ) \right]_{q^{[C]}}& = \left[\log Z^{\PT}(X,\inv )
\right]_{q^{[C]}}\\
& = \sum_{g',h'} \nPT{g',h'}{[C]}
\psi^{h'-1}_{y}\psi^{g'+1-2h'}_{w}. 
\end{align*}
On the other hand, by Proposition~\ref{prop: iPT and iMT moduli spaces for degree 1
local curve} and using the fact that the Behrend function is
$(-1)^{d}$ on a smooth scheme of dimension $d$ we get 
\begin{align*}
\left[Z^{\PT}(X,\inv ) \right]_{q^{[C]}} & = \sum_{T\subset \{1,\dotsc
,2m \}}\sum_{n} (-1)^{n+h-1} e\left(\Sym^{n+h-1}(C/\inv) \right)\,
y^{n} w^{|T|-m}\\
&=\left(\sum_{k=0}^{2m}\binom{2m}{k}w^{k-m} \right)
y^{1-h}\sum_{d=0}^{\infty}e\left(\Sym^{d}(C/\inv ) \right) (-y)^{d} \\
&=\psi_{w}^{m}\,y^{1-h}(1+y)^{2h-2}\\
&=\psi_{y}^{h-1}\,\psi_{w}^{g+1-2h}
\end{align*}
where we used MacDonald's formula \cite{Macdonald-poincare-poly} for
the penultimate equality. The formula for $\nPT{g',h'}{[C]}$ then
follows.

To compute $\nMT{g',h'}{[C]}$ we observe that $\Chow_{[C]}(X)$ is a
point and that $\M^{\epsilon}_{[C]}(X,\inv )$ is smooth. Consequently,
the Maulik-Toda polynomial is given by the (symmetrized) Poincare
polynomial:
\[
\MT_{[C]}(y,w) = \sum_{T\subset \{1,\dotsc 2m \}} w^{|T|-m} \,
\widetilde{P}_{y}\left(\Pic^{h}(C/\inv ) \right). 
\]
Since $\Pic^{h}(C/\inv )$ is an Abelian variety of dimension $h$, its
symmetrized Poincare polynomial is given by
$y^{-h}(1+y)^{2h}=\psi_{y}^{h}$. Thus
$\MT_{[C]}(y,w)=\psi_{y}^{h}\,\psi_{w}^{g+1-2h}$ and the formula for
$\nMT{g',h'}{[C]}$ is proved.
\end{proof}

\section{Local Abelian and Nikulin Surfaces (PT theory)}
\label{sec: Local Abelian/Nikulin Surfaces (PT theory)}

\subsection{Overview}\label{subsection:Overview} In this section we
compute $Z^{\PT}(X,\inv )$, and thus determine all the invariants
$\nPT{g,h}{\beta}$ for the case of $X=S\times \CC$, where $S$ is an Abelian surface with its
symplectic involution $\inv (a)=-a$ or a Nikulin $K3$ surface which by
definition comes with a symplectic involution (in both cases, $\inv$ acts trivially on the second factor). The main theorem is
given by Theorem~\ref{thm: main statement about nPT for
Abelian/Nikulin surfaces}.

Our basic tool for computing the $\inv$-PT invariants of $S\times \CC$
is the Donaldson-Thomas Crepant Resolution Conjecture (DT-CRC) which
was conjectured in \cite{Bryan-Cadman-Young} and recently proven by
Beentjes, Calabrese, and Rennemo in
\cite{Beentjes-Calabrese-Rennemo}. The idea is the following. Our
$\inv$-PT partition function $Z^{\PT} (X,\inv )$ can be written in
terms of the orbifold PT partition function $Z^{\PT}([X/\inv ])$ and
then the DT-CRC asserts that
\begin{equation}\label{eqn: ZPT(X)=ZPT(Y)/ZPTexc}
Z^{\PT}([X/\inv ]) = \frac{Z^{\PT}(Y)}{Z^{\PT}_{\exc}(Y)}
\end{equation}
where $Y\to X/\inv$ is the crepant resolution, $Z^{\PT}(Y)$ is the
ordinary PT partition function of $Y$, and $Z^{\PT}_{\exc} (Y)$ is the
partition function for curve classes supported on the exceptional
fibers. The variables in the above equality are identified via the
Fourier-Mukai isomorphism in numerical $K$-theory. In the case of
$X=S\times \CC$, 
\[
Y=\Sresolve \times \CC
\]
where $\Sresolve \to S/\inv$ is the minimal resolution. In the case
where $S$ is an Abelian surface, $\Sresolve$ is the associated Kummer
$K3$ surface, and in the case where $S$ is a Nikulin $K3$ surface,
$\Sresolve$ is a special kind of $K3$ surface which we call a Nikulin
resolution. 

We then can compute the right hand side of Equation~\eqref{eqn:
ZPT(X)=ZPT(Y)/ZPTexc} using the KKV formula (see Section~\ref{subsec:
KKV formula}). Doing this requires an explicit description of the
Picard lattice of $\Sresolve $, which was given by Garbagnati-Sarti
\cite{Garbagnati-Sarti-Kummer-surfaces-2016,Garbagnati-Sarti-even-set-2008}.
Finally, to complete the computation, we will need some theta function
identities which we prove in Section~\ref{subsec: Theta function identities}.

\subsection{Using the DT-CRC}\label{subsec: using the DT-CRC}

To use Equation~\eqref{eqn: ZPT(X)=ZPT(Y)/ZPTexc}, we will need to be
explicit with our choice of variables. Let $X=S\times \CC$ and let
$\beta_{d}$ be an effective, $\inv$-invariant, primitive curve class
on $S$ with $\beta_{d}^{2}=2d$ which we identify with the
corresponding class on $X$. To determine our invariants
$\nPT{g,h}{m\beta_{d}}$, we need to compute the partition function
\[
Z^{\PT} (X,\inv ) = \sum_{m\geq 0}\sum_{n, \epsilon}
N^{\PT}_{m\beta_{d},n,\epsilon}(X,\inv ) Q^{m}y^{n}w^{\epsilon}.
\]
\begin{remark}\label{rem: A Picard rank bigger than 1}
Strictly speaking, when the rank of the invariant Picard group $\Pic
(S)^{\inv }$ is greater than 1, the above is a \emph{restricted}
partition function: we don't sum over all invariant curve classes,
but only over the semi-group generated by $\beta_{d}$. This suffices
for determining the invariants $\nPT{g,h}{m\beta_{d}}$. A few of the
statements made in this section require minor adjustments in the case
where the invariant Picard rank is greater than 1. Note that for
$d\leq 0$, the invariant Picard rank is necessarily greater than 1.
\end{remark}

Our partition function $Z^{\PT}(X,\inv)$ can be determined from the PT
partition function of the orbifold $[X/\inv ]$. 
By definition (see \cite{Beentjes-Calabrese-Rennemo}), 
\[
Z^{\PT} ([X/\inv ]) = \sum_{\alpha \in N_{\leq 1}([X/\inv ])}
N_{\alpha}^{\PT}([X/\inv ]) \mathbf{Q}^{\alpha} 
\]
where the sum ranges over the numerical $K$-theory of
$\operatorname{Coh}_{\leq 1}([X/\inv ])$, the category of coherent
sheaves on $[X/\inv ]$ having proper support of dimension less than or
equal to one.

We need to choose generators for the free $\ZZ$-modules 
\[
N_{\leq 1}([X/\inv ])\cong N_{\leq 1}([S/\inv ])
\]
and
\[
N_{\leq 1}(Y)\cong N_{\leq 1}(\Sresolve )\cong H_{0}(\Sresolve )\oplus \Pic (\Sresolve )
\]
in a way that is compatible with the Fourier-Mukai isomorphism
\[
N_{\leq 1}([S/\inv ])\cong N_{\leq 1}(\Sresolve ).
\]

It will be convenient to choose generators over $\QQ$. Determining
which linear combinations of our generators are integral classes is
somewhat subtle and is addressed in Section~\ref{subsec: integral
classes on KmA}. The generators, and their corresponding variables in
the partition functions, are given in the following table:

\setlength{\arrayrulewidth}{0.5mm}
\setlength{\tabcolsep}{15pt}
\renewcommand{\arraystretch}{1.5}
\bigskip
\begin{table}[H]
\begin{center}
\begin{tabular}{lll}\toprule
Class in $N_{\leq 1}([S/\inv ])$ &
Class in $H_{0}(\Sresolve)\oplus \Pic (\Sresolve)$ & Variable \\
\midrule
$ [\OO_{\pt}]$ & $[\pt]$ & $y$\\
$[\OO_{x_{i}}\otimes R_{-}]$ & $[E_{i}]$ & $w_{i}$\\
$\alpha_{d}$  & $\gamma_{d}$ & $Q$\\ \bottomrule
\end{tabular}
   \end{center}
\end{table}
\bigskip

Here $\OO_{\pt}$ is the structure sheaf of a generic
point on $[S/\inv ]$,
$\OO_{x_{i}}\otimes R_{-}$ is the structure sheaf of the $i$-th
orbifold point $x_{i}$ equipped with the non-trivial action of its
stabilizer group, and
\[
\alpha_{d}  = \half t_{*}( \operatorname{ch}^{-1}(\beta_{d}))
\]
where $t:S\to [S/\inv] $ and $\operatorname{ch}^{-1}(\beta_{d})$ is
the class in $N_{\leq 1}(S)$ corresponding to $\beta_{d}\in
\Pic (S)$ under the Chern character isomorphism. The
generators of $H_{0}(\Sresolve)$ and $ \Pic (\Sresolve)$ are given by
the point class $[\pt ]$, the classes of the exceptional divisors
$E_{i}$, and
\[
\gamma_{d} = c_{1}(FM(\alpha_{d})),
\]
the divisor class associated to the image of $\alpha_{d}$ under the
Fourier-Mukai isomorphism.

The fact that the above choices are compatible with the Fourier-Mukai
isomorphism uses the well-known fact that the isomorphism takes
$\OO_{x_{i}}\otimes R_{-}$ to $\OO_{E_{i}}(-1)$.

With these variables, Equation~\eqref{eqn: ZPT(X)=ZPT(Y)/ZPTexc} can
be viewed as an equality of formal series\footnote{In general, the
statement of the DT-CRC requires viewing the partition functions as
rational functions in certain variables and the equality as an
equality of rational functions. This issue does not arise in this
case. } in the variables $y,w_{i},Q$.

\begin{lemma}\label{lem: ZPT(X,i)=ZPT([X/i])|wi=w}
$Z^{\PT}(X,\inv )=\restriction{Z^{\PT}([X/\inv ])}{w_{i}=w}$
\end{lemma}
\begin{proof}
Our $\inv$-PT pairs on $X$ are equivalent to PT pairs on the stack
quotient $[X/\inv ]$. However, keeping track of the $K$-theory class
of the sheaf on $[X/\inv ]$ is a refinement of the discrete data used
for $\inv$-PT pairs. The lemma follows from observing that a sheaf on
$[X/\inv ]$ in the $K$-theory class
\[
m\alpha_{d} +n[\OO_{\pt}]+\sum_{i} v_{i}[\OO_{x_{i}}\otimes R_{-}]
\]
corresponds to an $\inv$-equivariant sheaf $F$ with $[\supp
(F)]=m\beta_{d}$ and 
\[
 \chi (F) = nR_{\reg}+(\sum_{i}v_{i})R_{-} .
\]
\end{proof}

Next we define the \emph{exceptional lattice} 
\[
\Lambda =\oplus_{i}\ZZ\left\langle E_{i} \right\rangle \subset \Pic (\Sresolve ) 
\]
and we define
\[
 \Gamma_{m,d}=\left\{v\in \Lambda \otimes \QQ:  \!\!\!\!  \quad \text{$ m\gamma _{d}
 +v$ is a non-zero integral class in }\Pic (\Sresolve) \right\}. 
\]

For $v=\sum_{i}v_{i}E_{i}$ we will use the following notation
\[
l(v)=\sum_{i}v_{i}, \quad
v^{2}=-2\sum_{i}v^{2}_{i}. 
\]

We can write the log of the partition function on $Y$ in terms of the
Gopakumar-Vafa invariants of $Y$ using Equation~\eqref{eqn:
log(ZPT)=sum 1/k etc}. We then specialize $w_{i}$ to $w$ to get:
\[
\restriction{\log Z^{\PT}(Y)}{w_{i}=w}
 = \sum_{k>0}\sum_{m\geq 0}\sum_{v\in \Gamma_{m,d}} \sum_{h\geq 0}
 \frac{1}{k}\nPT{h}{m\gamma_{d}+v }
 Q^{km}w^{kl(v)}\psi^{h-1}_{-(-y)^{k}}.
\]
On the other hand, the invariants $\nPT{g,h}{m\beta_{d}}$ on $X$ are
by definition given by
\[
\log Z^{\PT}(X,\inv ) = \sum_{k,m>0}\sum_{g,h} \frac{1}{k}Q^{km}
\nPT{g,h}{m\beta_{d}} \psi^{h-1}_{-(-y)^{k}}\psi^{g+1-2h}_{w^{k}}.
\]

Taking the log of Equation~\eqref{eqn: ZPT(X)=ZPT(Y)/ZPTexc},
observing that $Z^{\PT}_{\exc}(Y)=\restriction{Z^{\PT}(Y)}{Q=0}$, and
applying Lemma~\ref{lem: ZPT(X,i)=ZPT([X/i])|wi=w}, we
get
\[
\log Z^{\PT}(X,\inv ) = \restriction{\log Z^{\PT}(Y)}{w_{i}=w} -
\restriction{\log Z^{\PT}(Y)}{w_{i}=w,Q=0}. 
\]
Combining this with the previous two equations, we arrive at
\[
\begin{split}
\sum_{k,m>0} \frac{Q^{km}}{k} & \left(\sum_{g,h} \nPT{g,h}{m\beta_{d}} \psi^{h-1}_{-(-y)^{k}} \psi^{g+1-2h}_{w^{k}} \right) \\
& =\sum_{k,m>0}\frac{Q^{km}}{k} \left(\sum_{h\geq 0}\sum_{v\in \Gamma_{m,d}} \nPT{h}{m\gamma_{d}+v } w^{kl(v)} \psi^{h-1}_{-(-y)^{k}} \right). 
\end{split}
\]

By M\"{o}bius inversion (or a simple induction argument), the quantities
in the parenthesis in the above equation must be equal for all $k $
and $m$. In particular, by setting $k=1$ we've proved
\begin{equation}\label{eqn: nPTgh on X in terms of nPT on Y}
\sum_{g,h} \nPT{g,h}{m\beta_{d}}\psi^{h-1}_{y} \psi^{g+1-2h}_{w} =
\sum_{h\geq 0}\sum_{v\in \Gamma_{m,d}} \nPT{h}{m\gamma_{d}+v } w^{l(v)}
\psi^{h-1}_{y}.
\end{equation}

The invariants $\nPT{h}{m\gamma_{d}+v }$ are determined by the KKV
formula since $Y=\Sresolve\times \CC$ is a local $K3$ surface. One
formulation of the KKV formula from Section~\ref{subsec: KKV formula}
is the following. For any effective curve class $C$, $n_{h}(C)$ is
given by
\[
\sum_{h\geq 0}n_{h}(C) \psi_{y}^{h-1} = 
\left[\frac{1}{\phi_{10,1}(q,-y)} \right]_{q^{\frac{C^{2}}{2}}} 
\]
where
\begin{equation*}
\phi_{10,1}(q,-y) =-\psi_{y}\cdot
q\prod_{n=1}^{\infty}(1+yq^{n})^{2}(1+y^{-1}q^{n})^{2}(1-q^{n})^{20}  
\end{equation*}
and where $\left[\dotsb \right]_{q^{a}}$ denotes the coefficient of
$q^{a}$ in the expression $[\dotsb]$ . Applying this to
Equation~\eqref{eqn: nPTgh on X in terms of nPT on Y} and using the
facts that $\gamma_{d}^{2}=d$ and $\gamma_{d}\in \Lambda^{\perp}$ we get
\[
(m\gamma_{d}+v )^{2} = m^{2}d+v^{2}
\]
and so
\begin{align*}
\sum_{g,h} \nPT{g,h}{m\beta_{d}} \psi_{y}^{h-1}\psi_{w}^{g+1-2h} =&
\sum_{v\in \Gamma_{m,d}} \left[\frac{w^{l(v)}}{\phi_{10,1}(q,-y)} \right]_{q^{\half (m^{2}d+v^{2})}}\\
=&\left[\frac{\Theta_{\Gamma_{m,d}}(q,w)}{\phi_{10,1}(q,-y)}
\right]_{q^{\half m^{2}d}}
\end{align*}
where for any subset $T\subset \Lambda \otimes \QQ$ we've defined 
\[
\Theta_{T}(q,w) = \sum_{v\in T}
q^{-\frac{v^{2}}{2}}w^{l(v)}. 
\]

In Section~\ref{subsec: integral classes on KmA} we compute
$\Gamma_{m,d}$. The results are:
\begin{proposition}\label{prop: the lattice Gammamd}
The subset $\Gamma_{m,d}\subset \Lambda \otimes \QQ$ is given as follows:
\begin{itemize}
\item If $S$ is an Abelian surface, or a Type (I) Nikulin surface, then
\[
\Gamma_{m,d}=
\begin{cases}
L&\text{if $m$ is even}\\
L+ r_{0}&\text{if $m$ is odd and $d$ is even}\\
L+ r_{1}&\text{if $m$ is odd and $d$ is odd}\\
\end{cases}
\]
where in the Abelian surface case $L=K $,
the so-called Kummer lattice, an even, negative definite rank 16
lattice, and in the Type (I) Nikulin case, $L=N$
is the so-called Nikulin lattice, an even, negative definite rank 8
lattice. The vectors $r_{0}$ and $r_{1}$ are particular vectors we will define
in Section~\ref{subsec: integral classes on KmA}. See
Section~\ref{subsec: integral classes on KmA} for the definition of
$K$ and $N$.
\item If S is a Type (II) Nikulin surface, then $d$ is even and
\[
\Gamma_{m,d} = N
\]
where $N$ is the Nikulin lattice.
\end{itemize}
\end{proposition}

The shifted lattices $L+r_{i}$ have the
property that all their vectors have squares which are congruent to $i$
modulo 2.  It follows that we may write
\[
\Theta_{L+ r_{0}}(q,w)+\Theta_{L+ r_{1}}(q,w) =
\Theta_{L_{\sh}}(q,w)
\]
where 
\[
L_{\sh }= (L+r_{0})\cup (L+r_{1}).
\]

In summary, we've shown that
\begin{equation}\label{eqn: formula for nPTgh(mbetad)}
\sum_{g,h} \nPT{g,h}{m\beta_{d}} \psi^{h-1}_{y}\psi_{w}^{g+1-2h}
=\left[\frac{\Theta_{T}(q,w)}{\phi_{10,1}(q,-y)} \right]_{q^{\half
m^{2}d}}
\end{equation}
where
\begin{equation}\label{eqn: Lattice type}
T=
\begin{cases}
 N&\text{if $S$ is Type (II) Nikulin, $\quad (\NiII)$}\\
N& \text{ if $S$ is Type (I) Nikulin and $m$ is even, $\quad (\evNiI)$}\\
 N_{\sh }&\text{if $S$ is Type (I) Nikulin and $m$ is odd, $\quad (\oddNiI )$}\\
 K&\text{if $S$ is Abelian and $m$ is even, $\quad (\evAb)$}\\
 K_{\sh }&\text{if $S$ is Abelian and $m$ is odd. $\quad (\oddAb )$}\\
\end{cases}
\end{equation}

We now make the crucial observation that the right hand side of the
Equation~\eqref{eqn: formula for nPTgh(mbetad)} only depends on the
curve class $m\beta_{d}$ through its square $(m\beta_{d})^{2} =
2m^{2}d$ and (possibly) its divisibility modulo 2 (i.e. $m \!\!\! \mod
2$). This leads to the main theorem of this section:

\begin{theorem}\label{thm: main statement about nPT for
Abelian/Nikulin surfaces} Let $\beta$ be \emph{any} effective
$\inv$-invariant curve class (not necessarily primitive) on an Abelian
or Nikulin surface $S$ with $\beta^{2}=2d$. Then the invariants
$\nPT{g,h}{\beta}$ of $S\times \CC$ only depend on $(g,h,d)$ in the
case where $S$ is a Type (II) Nikulin surface and on $(g,h,d)$ and
whether the divisibility of $\beta $ is odd or even in the other
cases. Denoting these invariants as $\nPT{g,h}{d;\type}$ where
$\type\in \{\oddAb ,\evAb ,\oddNiI ,\evNiI ,\NiII \}$. Then
\[
\sum_{g,h} \nPT{g,h}{d; \type } \psi_{y}^{h-1}\psi_{w}^{g+1-2h} =
\left[\frac{\Theta_{T}(q^{2},w)}{\phi_{10,1}(q^{2},-y)}
\right]_{q^{d}} .
\]
Moreover $\Theta_{T}(q^{2},w)$ is given explicitly by
\begin{equation}\label{eqn: formulas for ThetaT}
\Theta_{T}(q^{2},w) =
\begin{cases}
\theta_{0}^{16}+30\theta_{0}^{8}\theta_{1}^{8}+\theta_{1}^{16}
&\text{if $\quad T=K$,   (type  $\evAb $  ),}\\
&\\
\theta_{0}^{8}+\theta_{1}^{8}
&\text{if $\quad T=N$, (types $\evNiI$ and $\NiII $),}\\
&\\
4\cdot \frac{\Delta (q^{2})}{\Delta (q)^{2}}\cdot \phi_{10,1}^{2}(q,-w)
&\text{if $\quad T=K_{\sh}$, (type $\oddAb $),}\\
&\\
-\frac{\Delta (q^{2})^{\half}}{\Delta (q)}\cdot \phi_{10,1}(q,-w)
&\text{if $\quad T=N_{\sh}$, (type $\oddNiI $),}\\
\end{cases}
\end{equation}
where
\[
\Delta (q) = q\prod_{n=1}^{\infty} (1-q^{n})^{24}
\]
is the unique modular cusp form of weight 12, 
\[
\phi_{10,1}(q,y) =-\psi_{-y}\cdot
q\prod_{n=1}^{\infty}(1-yq^{n})^{2}(1-y^{-1}q^{n})^{2}(1-q^{n})^{20}  
\]
is the
unique Jacobi cusp form of weight 10 and index 1, and
\[
\theta _{i} = \theta_{i}(q^{2},w) = \sum_{k\in \ZZ +\frac{i}{2}} q^{2k^{2}}w^{k}
\]
are the standard rank 1 theta functions.

We note that for $\nPT{g,h}{d;\evAb}$ and $\nPT{g,h}{d;\evNiI}$, $d$
is necessarily divisible by 4, and for $\nPT{g,h}{d;\NiII}$, $d$ is
necessarily even.
\end{theorem}

\begin{remark}
It is straightforward to see that the above formulas for the case of
$\oddAb$ and $\oddNiI$ lead to the product formulation given in
Theorem~\ref{thm: Formulas for nPTgh for X=SxC} in the Introduction.
\end{remark}

\begin{remark}
The theta functions $\Theta_{K}(q^{2},w)$ and
$\Theta_{K_{\sh}}(q^{2},-w)$ are Jacobi forms of weight 8 and index 2
(for some congruence subgroup), while the theta functions
$\Theta_{N}(q^{2},w)$ and $\Theta_{N_{\sh}}(q^{2},-w)$ are Jacobi
forms of weight 4 and index 1 (for some congruence subgroup). It would
be nice to have a direct, lattice theoretic explination of the
identity $\Theta_{K_{\sh}}=4\Theta_{N_{\sh}}^{2}$.
\end{remark}

To complete the proof of Theorem~\ref{thm: main statement about nPT
for Abelian/Nikulin surfaces}, we must prove Proposition~\ref{prop:
the lattice Gammamd} and we must prove the formulas for
$\Theta_{T}(q^{2},w)$ given by Equation~\eqref{eqn: formulas for
ThetaT}. This is carried out in the next two subsections.

\subsection{The Picard lattice of $\Sresolve$.}\label{subsec: integral classes on KmA}

Recall that $S$ is an Abelian or Nikulin surface and $\Sresolve \to
S/\inv$ is the associated Kummer $K3$ or Nikulin resolution
respectively. In this section we describe $\Pic (\Sresolve )$ and in
particular prove Proposition~\ref{prop: the lattice Gammamd}.  Recall
also that we defined the exceptional lattice:
\[
\Lambda =\oplus_{i}\ZZ\left\langle  E_{i} \right\rangle\subset \Pic (\Sresolve ).
\]
\begin{definition}\label{defn: the lattice L (which is either K or N)}
Let $L\subset \Pic (\Sresolve )$ be the saturation of $\Lambda $ in
$\Pic (\Sresolve )$, i.e. the smallest primitive sublattice containing
$\Lambda$ such that $\Lambda$ generates $L$ over $\QQ$. If $\Sresolve$
is a Kummer surface, then $L$ is by definition $K$, the Kummer
lattice.  If $\Sresolve$ is a Nikulin resolution, then $L$ is by
definition $N$, the Nikulin lattice.
\end{definition}

We note that by construction, we have the inclusions
\[
\Lambda \subset L\subset L^{\vee}\subset \Lambda^{\vee}=\half \Lambda .
\]
Explicit descriptions of $K$ and $N$ are given in the following
lemmas. The first is due to Nikulin, see for example \cite[Lemma
5.2]{Morrison-K3surfaces-1984} .

\begin{lemma}\label{lem: description of N} The Nikulin lattice $N$ is
the overlattice of $\Lambda$ generated by $\Lambda$ and $\Ehat =\half
\sum_{i}E_{i}$.
\end{lemma}

While the above shows that the Nikulin lattice is obtained from the
exceptional lattice by adding a single vector, the situation for the
Kummer lattice is more complicated. The pithiest way to state the
result is as follows (see \cite[\S~VIII.5]{Compact-Complex-Surfaces-BHPV})
\begin{lemma}\label{lem: description of K}
Under the natural identification
\[
\Lambda^{\vee}/\Lambda
\cong \operatorname{Maps}(\FF_{2}^{4},\FF_{2}),
\]
the Kummer lattice $K$ is the overlattice of $\Lambda$ such that the
inclusion 
\[
K/\Lambda \subset \Lambda^{\vee}/\Lambda 
\]
corresponds to the inclusion
\[
\operatorname{Aff}(\FF_{2}^{4},\FF_{2})\subset \operatorname{Maps}(\FF_{2}^{4},\FF_{2})
\]
of affine linear maps (including the two constant maps) into all maps.
\end{lemma}

We next describe $\Pic (\Sresolve )$ which will allow us to prove
Proposition~\ref{prop: the lattice Gammamd}. The class $\gamma_{d}$
might not be an integral class, but it turns out that the embedding
\[
\ZZ \left\langle 2\gamma_{d} \right\rangle \oplus L \subset \Pic (\Sresolve )
\]
is always index two.  Depending on the parity of $d$ and the type of
the surface $S$, the order two quotient group is generated either by
$\gamma_{d}$, or by $\gamma_{d}+r_{d}$ where $r_{d}$ is a certain
vector only depending on $d \!\!\! \mod 2$.

In the case where $S$ is Abelian and $\Sresolve$ is a Kummer $K3$, we
quote Garbagnati-Sarti \cite[Theorem~2.7]{Garbagnati-Sarti-Kummer-surfaces-2016}
adapted to our notation.

\begin{proposition}\label{prop: Pic of Kummer}
The Picard lattice of the Kummer surface $\Sresolve$ is the index 2
overlattice of $\ZZ \left\langle 2\gamma_{d} \right\rangle \oplus K$
generated by  $\ZZ \left\langle 2\gamma_{d} \right\rangle \oplus K$
and $\gamma_{d}+r_{d}$ where
\begin{itemize}
\item $r_{d}\in K^{\vee}-K$, $2r_{d}\in K$,
\item $r^{2}_{d}=d \!\!\! \mod 2$. 
\end{itemize}
The class $[r_{d}]\in \Lambda^{\vee}/\Lambda \cong
\operatorname{Maps}(\FF_{2}^{4},\FF_{2})$ only depends on $d \!\!\! \mod 2$
and is unique up to isometries of $K$. For $d$ even, the corresponding
map $\FF_{2}^{4}\to \FF_{2}$ is the characteristic function of a fixed
linear 2 plane $P_{1}\subset \FF_{2}^{4}$. For $d$ odd, the
corresponding map is the characteristic function of $P_{1}\Delta P_{2}
\subset \FF_{2}^{4}$ where $P_{1}$ and $P_{2}$ are transversely
intersecting 2 planes, and $\Delta$ denotes symmetric difference.
\end{proposition}

There are two families of Nikulin $K3$ surfaces determined as
follows. Let $S$ be a Nikulin surface and recall that $\beta_{d}\in
\Pic (S)$ is a primitive $\inv$-invariant effective class with
$\beta_{d}^{2}=2d$. The existence of the Nikulin involution implies
there is an inclusion
\[
\ZZ \left\langle \beta_{d} \right\rangle\oplus E_{8}(-2)\subset \Pic (S).
\]
The above is either (I) an isomorphism, or (II) an index 2
sublattice\footnote{This statement must be modified if the
invariant Picard rank of $S$ is greater than 1. c.f. Remark~\ref{rem:
A Picard rank bigger than 1} .}.

\begin{definition}\label{defn: Type I and Type II}
We say that $S$ is Nikulin of Type
(I) in the first case and of Type (II) in the second case. The latter
can occur only when $d$ is even.
\end{definition}

\begin{proposition}\label{prop: Picard lattice of a Nikulin resolution}
The Picard lattice of a Type (II) Nikulin resolution $\Sresolve$ is $\ZZ
\left\langle \gamma_{d}  \right\rangle \oplus N$. The Picard lattice
of a Type (I) Nikulin resolution $\Sresolve$ is the index 2 over
lattice of  $\ZZ
\left\langle 2\gamma_{d}\right\rangle  \oplus N $ generated by  $\ZZ
\left\langle 2\gamma_{d}\right\rangle  \oplus N $ and $\gamma_{d}+r_{d}
$ where
\begin{itemize}
\item $r_{d}\in N^{\vee}-N$, $2r_{d}\in N$,
\item $r_{d}^{2}=d \!\!\! \mod 2$. 
\end{itemize}
The class $[r_{d}]\in \Lambda^{\vee}/\Lambda$ only depends on $d \!\!\! \mod 2$
and is unique up to isometries of $N$ and is given by
\[
r_{d}=\begin{cases}
\half (E_{1}+E_{2}) & \text{if $d$ is odd}\\
\half (E_{1}+E_{2}+E_{3}+E_{4}) & \text{if $d$ is even}\\
\end{cases}
\]
for a suitable numbering of the exceptional divisors $E_{1},\dotsc ,E_{8}$.
\end{proposition}
\proof See Proposition~2.1 and Corollary~2.2 of
\cite{Garbagnati-Sarti-even-set-2008}.

Propositions~\ref{prop: Pic of Kummer} and \ref{prop: Picard lattice
of a Nikulin resolution} then prove Proposition~\ref{prop: the lattice Gammamd}.

\subsection{Theta function identities}\label{subsec: Theta function identities}

To finish the proof of Theorem~\ref{thm: main statement about nPT for
Abelian/Nikulin surfaces}, we must prove the formulas given in
Equation~\eqref{eqn: formulas for ThetaT}. Recall that
\[
\Lambda \subset L\subset L^{\vee}\subset \Lambda ^{\vee }=\half \Lambda 
\]
where $\Lambda =\oplus_{i}\ZZ \left\langle E_{i} \right\rangle$ is the
exceptional lattice and $L$ is either $K$ or $N$. Since any element
$\rho \in \Lambda^{\vee}/\Lambda$ may be uniquely written as
\[
\rho =\half\sum_{i} \rho_{i}E_{i}, \quad \rho_{i}\in \{0,1 \}
\]
we may define
\[
c_{1}(\rho ) = \sum_{i} \rho_{i}, \quad c_{0}(\rho ) = \operatorname{rk}(\Lambda )-c_{1}(\rho ),
\]
i.e. the number of $\rho_{i}$'s which are 1 or 0 respectively. The
following lemma is our basic tool for computing theta
functions\footnote{The authors are very grateful to John Duncan who
explained this to us.}
\begin{lemma}\label{lem: coset method for computing ThetaT}
Let $\pi$ be the projection $\Lambda^{\vee}\to \Lambda^{\vee}/\Lambda$
and suppose that $T\subset \Lambda^{\vee}$ is a union of cosets:
$T=\cup_{\rho \in \pi (T)}(\Lambda +\rho )$. Then
\[
\Theta_{T}(q^{2},w) = \sum_{\rho \in \pi (T)} \theta_{0}^{c_{0}(\rho
)} \theta_{1}^{c_{1}(\rho )}
\]
where
\[
\theta_{i} =\theta_{i}(q^{2},w)= \sum_{k\in \ZZ +\frac{i}{2}} q^{2k^{2}}w^{k}. 
\]
\end{lemma}

\begin{proof}
Since the cosets $\Lambda +\rho$ are disjoint we have 
\[
\Theta_{T}(q^{2},w) = \sum_{\rho \in \pi (T)} \Theta_{\Lambda +\rho} (q^{2},w)
\]
and then we observe that 
\[
\Lambda +\rho \cong \ZZ \left\langle E \right\rangle^{\oplus
c_{0}(\rho )} \oplus \left(\ZZ \left\langle E \right\rangle
+\frac{E}{2} \right)^{\oplus c_{1}(\rho )} 
\]
from which it follows that $\Theta_{\Lambda
+\rho}=\theta_{0}^{c_{0}(\rho )}\theta_{1}^{c_{1}(\rho )}$. 
\end{proof}

\begin{proposition}\label{prop: formulas for ThetaN and ThetaNsh}
The theta functions of the Nikulin lattice $N$ and the shifted Nikulin
lattice $N_{\sh}$ are given by
\begin{align*}
\Theta_{N}(q^{2},w) &= \theta_{0}^{8}+\theta_{1}^{8}\\
\Theta_{N_{\sh}}(q^{2},w) &= \theta_{0}^{2}\theta_{1}^{6}+2
\theta_{0}^{4}\theta_{1}^{4}+
\theta_{0}^{6}\theta_{1}^{2}\\ 
\end{align*}
\end{proposition}
\begin{proof}
It follows from Lemma~\ref{lem: description of N} and
Proposition~\ref{prop: Picard lattice of a Nikulin resolution} that
\[
\pi (N) = \left\{0,\tinyhalf (E_{1}+\dotsb +E_{8}) \right \} 
\]
and that
\begin{align*}
\pi (N_{\sh}) &= \pi \left(N+\tinyhalf (E_{1}+E_{2})  \right) \cup  \pi \left(N+\tinyhalf (E_{1}+\dotsb +E_{4}) \right)\\
&= \left\{\tinyhalf (E_{1}+E_{2}),\tinyhalf (E_{3}+\dotsb +E_{8}),
\tinyhalf (E_{1}+\dotsb +E_{4}),\tinyhalf (E_{5}+\dotsb +E_{8})
\right\}.
\end{align*}
The value of $c_{1}$ on the above 4 elements is 2, 6, 4, and 4
respectively.  The proposition then follows from Lemma~\ref{lem: coset
method for computing ThetaT}.
\end{proof}

\begin{proposition}\label{prop: formulas for ThetaK and ThetaKsh}
The theta functions of the Kummer lattice $K$ and the shifted Kummer
lattice $K_{\sh}$ are given by
\begin{align*}
\Theta_{K}(q^{2},w)&= \theta_{0}^{16}+ 30 \theta_{0}^{8}\theta_{1}^{8}
+\theta_{1}^{16}  \\
\Theta_{K_{\sh}}(q^{2},w)&=  4 \theta_{0}^{4}\theta_{1}^{12}+ 16
\theta_{0}^{6}\theta_{1}^{10}+ 24 \theta_{0}^{8}\theta_{1}^{8}+ 16
\theta_{0}^{10}\theta_{1}^{6}+  4 \theta_{0}^{12}\theta_{1}^{4}
\end{align*}
\end{proposition}
\begin{proof}
As in the Nikulin case, we must determine the value of $c_{1}$ on all
the elements of $\pi (K)$ and $\pi (K_{\sh})$. By Lemma~\ref{lem:
description of K}, $\pi (K)$ is given by the 32 elements $\half
\sum_{i}\rho_{i}E_{i}$ where the corresponding map $\FF_{2}^{4}\to
\FF_{2}$ given by $i\mapsto \rho_{i}$ is an affine linear
function. The value of $c_{1}$ on the two constant functions is 0 and
16 respectively, while the value of $c_{1}$ on the remaining 30
non-constant affine linear functions is 8. The formula for
$\Theta_{K}$ then follows from Lemma~\ref{lem: coset method for
computing ThetaT}. A simple but tedious way to determine the elements
of $\pi (K_{\sh}) = \pi (K+r_{0})\cup \pi (K+r_{1})$ is to choose
coordinates for $\FF_{2}^{4}$ and write down all 64 elements
explicitly. Doing so and reading the off the value of $c_{1}$ on each
element we find that $\pi (K+r_{0})$ has 4 elements with $c_{1}=4$, 24
elements with $c_{1}=8$, and 4 elements with $c_{1}=12$ and that $\pi
(K+r_{1})$ has 16 elements with $c_{1}=6$ and 16 elements with
$c_{1}=10$. The formula for $\Theta_{K_{\sh}}$ then follows from
Lemma~\ref{lem: coset method for computing ThetaT}. There is a more
coordinate free approach to the same calculation using the affine
geometry of $\FF_{2}^{4}$. It requires analyzing the various symmetric
differences of the affine hyperplanes and the two dimensional planes
$P_{1}$ and $P_{2}$ appearing in Proposition~\ref{prop: Pic of
Kummer}. Unfortunately, the case by case analysis in this approach is
not particularly less tedious than the direct enumeration.
\end{proof}
    
To complete the proof of Theorem~\ref{thm: main statement about nPT for
Abelian/Nikulin surfaces}, it only remains to prove the following
identities 
\begin{align*}
\Theta_{N_{\sh}}(q^{2},w)& = -\frac{\Delta (q^{2})^{\half}}{\Delta (q)}
\cdot \phi_{10,1}(q,-w) \\
&= \psi_{w}\cdot q\cdot \prod_{n=1}^{\infty}
(1+q^{n})^{12}(1-q^{n})^{8}(1+wq^{n})^{2} (1+w^{-1}q^{n})^{2}
\end{align*}
and
\begin{align*}
\Theta_{K_{\sh}}(q^{2},w)& = 4\frac{\Delta (q^{2})}{\Delta (q)^{2}}
\cdot \phi^{2}_{10,1}(q,-w) \\
&= 4\psi^{2}_{w}\cdot q^{2}\cdot \prod_{n=1}^{\infty}
(1+q^{n})^{24}(1-q^{n})^{16}(1+wq^{n})^{4} (1+w^{-1}q^{n})^{4}.
\end{align*}

Since the equations for $\Theta_{N_{\sh}}(q^{2},w)$ and
$\Theta_{K_{\sh}}(q^{2},w)$ given in Propositions~\ref{prop: formulas
for ThetaN and ThetaNsh} and \ref{prop: formulas for ThetaK and
ThetaKsh} can be factored as
\begin{align*}
\Theta_{N_{\sh}}(q^{2},w) &=
\theta_{0}^{2}\theta_{1}^{2}(\theta_{0}^{2}+\theta_{1}^{2})^{2} \\
\Theta_{K_{\sh}}(q^{2},w) &=
4\theta_{0}^{4}\theta_{1}^{4}(\theta_{0}^{2}+\theta_{1}^{2})^{4} 
\end{align*}
we see that it suffices to prove the identity:
\begin{equation}\label{eqn: product formula identity for Ksh}
\theta_{0}^{2}\theta_{1}^{2}(\theta_{0}^{2}+\theta_{1}^{2})^{2}
=\psi_{w}\cdot q\cdot \prod_{n=1}^{\infty} (1+q^{n})^{12}(1-q^{n})^{8}
(1+wq^{n})^{2} (1+w^{-1}q^{n})^{2}.
\end{equation}

By the Jacobi triple product identity, we may write
\begin{align}\label{eqn: Jacobi triple product for theta0 and theta1}
\theta_{0}(q^{2},w)&= \prod_{n=1}^{\infty}
(1-q^{4n})(1+wq^{4n-2})(1+w^{-1}q^{4n-2}) \\
\theta_{1}(q^{2},w)&=q^{\half}\cdot (w^{\half}+w^{-\half})\cdot  \prod_{n=1}^{\infty}
(1-q^{4n})(1+wq^{4n})(1+w^{-1}q^{4n}) .
\end{align}

We also have the following
\begin{lemma}\label{lem: quadratic theta function identity}
\[
\theta_{0}(q^{2},w)^{2}+\theta_{1}(q^{2},w)^{2} = \prod_{n=1}^{\infty}
(1-q^{2n})^{2} (1+q^{2n-1})^{2}(1+wq^{2n-1})(1+w^{-1}q^{2n-1}). 
\]
\end{lemma}
\begin{proof}
The left hand side of the above equation is given by
\[
\sum_{n,m\in \ZZ} q^{2m^{2}+2n^{2}} w^{n+m} +q^{2\left(m+\tinyhalf
\right)^{2}+2\left(n+\tinyhalf  \right)^{2}}w^{n+m+1}. 
\]
Letting $n=\tinyhalf (a-b)$ and $m=\tinyhalf (a+b)$ the sum rearranges
to
\[
\sum_{\substack{a,b\in \ZZ\\a\equiv b\bmod 2 }}
q^{b^{2}}\left(q^{a^{2}}w^{a}+q^{(a+1)^2}w^{a+1} \right) =
\sum_{a,b\in \ZZ} q^{b^{2}}q^{a^{2}}w^{a} =
\theta_{0}(q,1)\theta_{0}(q,w). 
\]
Then applying the Jacobi triple product identity to the right hand side of
the above proves the lemma.
\end{proof}

Now applying Lemma~\ref{lem: quadratic theta function identity} and
Equations~\eqref{eqn: Jacobi triple product for theta0 and theta1} to
the left hand side of Equation~\eqref{eqn: product formula identity
for Ksh}, we get 
\begin{align*}
\theta_{0}^{2}\theta_{1}^{2} (\theta_{0}^{2}+\theta_{1}^{2})^{2} =&
q\cdot  \psi_{w}\cdot \prod_{n=1}^{\infty} (1-q^{4n})^{2}(1+wq^{4n})^{2}(1+w^{-1}q^{4n})^{2}\\
&\quad \quad \quad \cdot (1-q^{4n})^{2}(1+wq^{4n-2})^{2}(1+w^{-1}q^{4n-2})^{2}\\
&\quad \quad \quad \cdot (1-q^{2n})^{4}(1+q^{2n-1})^{4}(1+wq^{2n-1})^{2}(1+w^{-1}q^{2n-1})^{2}\\
=& q\cdot\psi_{w}\cdot \prod_{n=1}^{\infty} (1+q^{n})^{12}(1-q^{n})^{8}
(1+wq^{n})^{2} (1+w^{-1}q^{n})^{2}
\end{align*}
where in the last equality we have used the fact that
\[
\prod_{n=1}^{\infty}(1+wq^{4n})(1+wq^{4n-2})(1+wq^{2n-1}) =
\prod_{n=1}^{\infty}(1+wq^{n})
\]
and similar considerations.

This completes the proof of Theorem~\ref{thm: main statement about nPT for
Abelian/Nikulin surfaces}. 

\section{Local Abelian and Nikulin Surfaces (MT theory)}
\label{sec: Local Abelian and Nikulin Surfaces (MT theory)}

\subsection{Overview.} In this section we prove some basic results
about $\inv$-stability and we prove Conjecture~\ref{conj: two tau-GV
defns are the same nPT=nMT} for local Abelian and Nikulin surfaces.

In Subsection~\ref{subsec: Nironi stability is equivalent to
i-stability} we show $\inv$-stability is equivalent to a certain kind
of Nironi stability on the stack quotient $[X/\inv ]$. In
Subsection~\ref{subsec: MT invariants for local Abeliand and Nikulin
surfaces} we prove Conjecture~\ref{conj: two tau-GV defns are the same
nPT=nMT} for $X=S\times \CC$ where $S$ is an Abelian or Nikulin
surface. The basic idea is the following. Using the results of
Subsection~\ref{subsec: Nironi stability is equivalent to i-stability}
and the projection $X\to \CC$ we show that all $\inv$-stable MT
sheaves on $X$ are given by Nironi $\delta$-stable sheaves on
$[S/\inv]\times \{t \}$ for some $t\in \CC$. We then show that Nironi
$\delta$-stability on $[S/\inv ]$ is the large volume limit of a
certain Bridgeland stability condition on $[S/\inv] $ constructed by
Lim and Rota \cite{Lim-Rota}. We then apply the derived Fourier-Mukai
correspondence to show that our moduli spaces are given by, up to a
factor of $\CC$, moduli spaces of objects in the derived category of
$\Sresolve$ which are stable with respect to the large volume limit of
one of the stability conditions on $K3$ surfaces constructed by
Bridgeland \cite{Bridgeland-K3-stability}. Finally, we use the results
of Bayer and Macri \cite{Bayer-Macri} to show these moduli spaces are
deformation equvalent to moduli spaces of MT sheaves on
$\Sresolve$. This then allows us to apply Conjecture~\ref{conj: nPTbg
= nDTbg}, which is known to hold for $\Sresolve \times \CC$ by
\cite{Pandharipande-Thomas-KKV,Shen-Yin-hyperkahler}. The upshot is
that we prove that Equation~\eqref{eqn: nPTgh on X in terms of nPT on
Y} holds with MT GV invariants replacing PT GV invariants on both
sides and then the subsequent arguments of Section~\ref{sec: Local
Abelian/Nikulin Surfaces (PT theory)} apply word for word.

\subsection{Nironi Stability}\label{subsec: Nironi stability is
equivalent to i-stability}

The category of $\inv$-equivariant sheaves on $X$ and the category of
sheaves on the stack $[X/\inv ]$ are canonically equivalent and in
this section we will not notationally differentiate between a sheaf on
the stack and the corresponding $\inv$-equivariant sheaf.

In \cite{Nironi-stability-on-stacks}, Nironi developed a theory of
slope stability for Deligne-Mumford stacks analougous to Simpson
stability for schemes. Nironi stability for the stack $[X/\inv ]$
involves a choice of an ample divisor $H$ on the coarse space $X/\inv$
and the choice of a ``generating bundle'' $V$ which we may take to be
(see \cite[Def.~2.2, Prop.~2.7]{Nironi-stability-on-stacks}) 
\[
V=\left(\OO_{X}\otimes R_{+} \right)^{\oplus a} \oplus
\left(\OO_{X}\otimes R_{-} \right)^{\oplus b}
\]
for any $a,b\in \NN$.

Nironi's slope function is obtained from the generalized Hilbert
polynomial of a sheaf $F$ (i.e. the $\inv$-invariant part of $\chi
(F\otimes V(mH))$) by dividing the second coefficient by the leading
coefficient. For 1-dimensional sheaves $F$ with $[\supp(F)]=\beta$,
$\chi (F) = nR_{\reg}+\epsilon R_{-}$, and our choice of $V$, Nironi's
slope function is given by
\[
\mu (F) = \frac{(a+b)n+b\epsilon}{(a+b)H\cdot \beta}. 
\]
The slope function only depends on $a$ and $b$ through the number
\[
\delta =\frac{b}{a+b}\in \QQ \cap (0,1)
\]
so we write
\[
\mu_{\delta}(F) = \frac{n+\delta \epsilon}{H\cdot \beta}. 
\]

\begin{definition}\label{defn: Nironi stability}
Let $F$ be an $\inv$-equivariant sheaf on $X$ with pure 1-dimensional
support, $[\supp(F)]=\beta$, and $\chi (F) = nR_{\reg}+\epsilon
R_{-}$. Then $F$ is \emph{Nironi $\delta$ (semi-)stable} if for all
$\inv$-equivariant subsheaves $F'\subsetneq F$, $\mu_{\delta}(F')<
\mu_{\delta}(F)$ (resp.  $\mu_{\delta}(F')\leq \mu_{\delta}(F)$).
\end{definition}

Let $\M^{\dstable}_{\beta ,n,\epsilon}([X/\inv ])$
(resp. $\M^{\dsemistable }_{\beta ,n,\epsilon}([X/\inv ])$) be the
moduli stack of Nironi $\delta$ (semi-)stable sheaves with $\beta
,n,\epsilon$ as above.  Nironi proves that the usual properties
enjoyed by moduli stacks of Simpson (semi-)stable sheaves hold for
moduli stacks of Nironi (semi-)stable sheaves \cite[Theorems 6.21,
6.22]{Nironi-stability-on-stacks}. In particular we have:

\begin{theorem}\label{thm: stack of Nironi stable sheaves is a C*
gerbe over its coarse space.}
The stack $\M^{\dstable}_{\beta ,n,\epsilon}([X/\inv ])$ is a
$\CC^{*}$-gerbe over its coarse moduli space. In particular, any
Nironi $\delta$ stable sheaf is simple. 
\end{theorem}

\begin{theorem}
Assume $X$ is projective, then $\M^{\dsemistable}_{\beta
,n,\epsilon}([X/\inv ])$ has a projective coarse moduli space. 
\end{theorem}

The following corollary is then standard.

\begin{corollary}\label{cor: Mnironi-->Chow is proper}
For $X$ quasi-projective, the Hilbert-Chow morphism 
\[
\M^{\dsemistable}_{\beta ,n,\epsilon}([X/\inv ])\to \Chow_{\beta}(X)^{\inv}
\]
given by $F \mapsto [\supp(F)]$ is proper. 
\end{corollary}

\begin{proposition}\label{prop: Nironi stability is equivalent to i-stability}
Let $F$ be an $\inv$-equivariant sheaf on $X$ having proper pure
1-dimensional support and with $[\supp(F)]=\beta$ and $\chi (F) =
R_{\reg}+\epsilon R_{-}$. Let $\delta>0$ be a sufficiently small
rational number. Then the following conditions are equivalent
\begin{enumerate}
\item $F$ is Nironi $\delta $ semistable.
\item $F$ is Nironi $\delta $ stable.
\item $F$ is $\inv$-stable (see Definition \ref{defn: tau-stable MT sheaves}).
\end{enumerate}
\end{proposition}
\begin{proof}
Let $F'\subsetneq F$ be an $\inv$-equivariant subsheaf with $\chi
(F')=kR_{\reg}+\gamma R_{-}$ and $[\supp(F')]=\beta '$.

Suppose that $F$ is Nironi $\delta$ semistable. Then the inequality
\[
\frac{k+\delta \gamma}{H\cdot \beta '} \leq \frac{1+\delta \epsilon}{H\cdot \beta}
\]
holds and is equivalent to 
\begin{equation}\label{eqn: stability inequality}
k\leq \left(\frac{H\cdot \beta '}{H\cdot \beta} \right) (1+\delta
\epsilon ) - \delta \gamma .
\end{equation}
Since $\supp(F') \subseteq \supp(F)$ and $\dim \supp(F') \neq 0$, we have 
\[
0< \frac{H\cdot \beta '}{H\cdot \beta}\leq 1.
\]
If $ \frac{H\cdot \beta '}{H\cdot \beta}< 1$, then for sufficiently
small $\delta$ we have $k<1$. If $ \frac{H\cdot \beta '}{H\cdot
\beta}=1$, then $\beta '=\beta$ and $k\leq 1+\delta (\epsilon -\gamma
)$.  Thus we see that either $k<1$ or $k=1$ and $\beta '=\beta$ and
$\gamma \leq \epsilon $. Finally, if $k=1$, $\beta '=\beta$, and
$\gamma =\epsilon$, then $\chi (Q)=0$ where
\[
0\to F'\to F\to Q\to 0.
\]
But since $\beta '=\beta$, $\dim Q=0$ and so $\chi (Q)=0$ implies that
$Q=0$ which implies that $F'=F$. Thus we've shown that (1) implies
that either $k< 1$ or $k=1$ and $\beta '=\beta$ and $\gamma < \epsilon
$ which by Definition~\ref{defn: tau-stable MT sheaves} means that $F$
is $\inv$-stable. Thus (1) implies (3). Moreover, we've shown that
the inequality~\eqref{eqn: stability inequality} must be strict so
that (1) implies (2). And of course (2) implies (1) so it remains to
show that (3) implies (2).

To that end, suppose that $F$ is $\inv$-stable, i.e. $k\leq 1$ and if
$k=1$ then $\gamma <\epsilon$ and $\beta '=\beta$. We need to prove
that the inequality~\eqref{eqn: stability inequality} holds. Since
$\dim \supp(F') > 0$, $\frac{H\cdot \beta '}{H\cdot \beta}>0$ which
means the right hand side of \eqref{eqn: stability inequality} is
positive for sufficiently small $\delta$, and so if $k<1$, \eqref{eqn:
stability inequality} is true. If $k=1$, then by hypothesis, $\gamma
<\epsilon$ and $\beta '=\beta$ so \eqref{eqn: stability inequality}
becomes $1\leq 1+\delta \epsilon -\delta \gamma$ which is true. 
\end{proof}

\subsection{Proof of Conjecture~\ref{conj: two tau-GV defns are the
same nPT=nMT} for Local Abelian and Nikulin surfaces}\label{subsec: MT
invariants for local Abeliand and Nikulin surfaces}

We now consider $X=S\times \CC$ with $S$ an Abelian or Nikulin surface
and we adopt the notation of Section~\ref{sec: Local Abelian/Nikulin
Surfaces (PT theory)}.

By Proposition~\ref{prop: Nironi stability is equivalent to
i-stability} we may identify the moduli space of $\inv$-stable MT sheaves
with the moduli space of Nironi $\delta$ stable sheaves:
\[
\M^{\epsilon}_{m\beta_{d}}(X,\inv ) = \M^{\dsemistable}_{m\alpha_{d}
,1,\epsilon}([X/\inv ])= \M^{\dstable}_{m\alpha_{d} ,1,\epsilon}([X/\inv ]).
\]

\begin{lemma}\label{lem: MT sheaves on SxC are supported on Sxt}
Let $F$ be an $\inv$-stable MT sheaf on $X$. Then $F$ is
scheme theoretically supported on $S\times \{t \}$ for some $t\in
\CC$. 
\end{lemma}
\begin{proof}
First suppose that the image of the support of $F$ under the map
$S\times \CC \to \CC$ is not a single point. Then $F$ can be written
as $F_{1}\oplus F_{2}$ which violates stability since both factors are
equivariant subsheaves. Thus we may suppose that $F$ is set
theoretically supported on some $S_{t}=S \times \{t \}$. Consider the
short exact sequence of $\inv$-equivariant sheaves:
\[
0\to \OO_{X}(-S_{t})\to \OO_{X} \to \OO_{S_{t}}\to 0.
\]
Noting that $\OO_{X}(-S_{t})\cong \OO_{X}$ and tensoring with $F$, we
get the right exact sequence:
\[
F\to F \to F\otimes \OO_{S_{t}} \to 0.
\]
By Theorem~\ref{thm: stack of Nironi stable sheaves is a C* gerbe over
its coarse space.}, $F$ is simple and hence the first map is either 0
or an isomorphism. Since $F\otimes \OO_{S_{t}}$ is non-zero by
construction, the first map must be zero and thus the second map
induces an isomorphism $F\cong F\otimes \OO_{S_{t}}$ by exactness.
\end{proof}

By the lemma, every $\inv$-stable MT sheaf on $X$ can be identified
with a Nironi $\delta$ stable sheaf on $[S/\inv ]\times \{t \}$ for
some $t\in \CC$. We let $\M^{\dstable}_{\eta}([S/\inv ])$ denote the
moduli space of Nironi $\delta$ stable sheaves on $[S/\inv ]$ in the
$K$-theory class $\eta$. Then applying the above discussion and the
analysis of $K$-theory classes done in Section~\ref{sec: Local
Abelian/Nikulin Surfaces (PT theory)} (and using the same notation) we
get 
\begin{equation}\label{eqn: M(X,i) = union M([S/i])xC}
\M^{\epsilon}_{m\beta_{d}}(X,\inv ) = \bigsqcup_{\substack{v\in \Gamma_{m,d}\\
 l(v)=\epsilon} }  \M^{\dstable}_{\eta(m\alpha_{d},v)}([S/\inv ])\times \CC 
\end{equation}
where
\[
\eta(m\alpha_{d},v) = m\alpha_{d} +[\OO_{\pt}] +
\sum_{i}v_{i}[\OO_{x_{i}}\otimes R_{-}]. 
\]

The following proposition is key:

\begin{proposition}\label{prop: Nironi stable sheaves in S/i is
deformation equivalent to MT sheaves on Shat} The moduli space
$\M^{\dstable}_{\eta (m\alpha_{d},v)}([S/\inv ])$ is deformation
equivalent to $\M^{s}_{(0,m\gamma_{d}+v,1)}(\Sresolve )$, the moduli
space of Simpson sheaves on $\Sresolve$ with Mukai vector
$(0,m\gamma_{d}+v,1)$. Moreover, the deformation equivalence is
compatible with the Hilbert-Chow morphism.
\end{proposition}

Assuming the above Proposition and noting that
\[
\M_{m\gamma_{d}+v}(\Sresolve \times \CC ) =
\M^{s}_{(0,m\gamma_{d}+v,1)}(\Sresolve)\times \CC,
\]
we may compute the
Maulik-Toda polynomial of $(X,\inv )$ as follows. We use the
definition, Equation~\eqref{eqn: M(X,i) = union M([S/i])xC}, and the
above Proposition to get
\[
\MT_{m\beta_{d}}(y,w) = \sum_{v\in \Gamma_{m,d}}\MT_{m\gamma_{d}+v}(y)w^{l(v)}
\]
where $\MT_{m\beta_{d}}(y,w) $ is the Maulik-Toda polynomial of
$(X,\inv )$ in the class $m\beta_{d}$ and $\MT_{m\gamma_{d}+v}(y)$ is
the Maulik-Toda polynomial of $\Sresolve \times \CC$ in the class
$m\gamma_{d}+v$.

It then follows immediately from the above equation that the MT analog
of Equation~\eqref{eqn: nPTgh on X in terms of nPT on Y} holds: 
\[
\sum_{g,h} \nMT{g,h}{m\beta_{d}}\psi^{h-1}_{y} \psi^{g+1-2h}_{w} =
\sum_{h\geq 0}\sum_{v\in \Gamma_{m,d}} \nMT{h}{m\gamma_{d}+v } w^{l(v)}
\psi^{h-1}_{y}.
\]

All the analysis in Section~\ref{sec: Local Abelian/Nikulin Surfaces
(PT theory)} subsequent to Equation (\ref{eqn: nPTgh on X in terms of
nPT on Y}) then goes through word for word with the MT versions of the
invariants and we see that they are given by the same formulas (in
Theorem~\ref{thm: main statement about nPT for Abelian/Nikulin
surfaces}) as the PT versions of the invariants. We thus conclude that 
\[
\nPT{g,h}{m\beta_{d}} = \nMT{g,h}{m\beta_{d}} 
\]
holds and thus Conjecture~\ref{conj: two tau-GV defns are the same nPT=nMT} holds for $(X,\inv
)$.

It remains only to prove Proposition~\ref{prop: Nironi stable sheaves
in S/i is deformation equivalent to MT sheaves on Shat}.

\begin{proof}[Proof of Proposition~\ref{prop: Nironi stable sheaves in S/i is
deformation equivalent to MT sheaves on Shat}] We remark that
although we don't directly use it, these moduli spaces are all
hyperkahler manifolds of $K3[n]$ type and the map to Chow is a Lagrangian
fibration.

In \cite{Lim-Rota}, Lim and Rota construct Bridgeland stability
conditions on orbifold surfaces with Kleinian orbifold points. For
notational simplicity, they assume that their orbifold surface has a
single orbifold point, but their method easily applies to orbifold
surfaces with multiple Kleinian orbifold points such as $[S/\inv ]$.

Their stability condition depends (in the case of $[S/\inv ]$) on
parameters $\gamma \in (0,\half )$ and $w\in \CC$ and has a central
charge $Z_{w,\gamma}$ which in our situation takes values
\[
Z_{w,\gamma}\left(m\alpha_{d} + n[\OO_{\pt}]+\sum_{i}v_{i}
[\OO_{x_{i}}\otimes R_{-}] \right) = -n-\half \gamma \sum_{i}v_{i} +i H\cdot
(m\alpha_{d}) .
\]

The parameter $w$ must be choosen satisfying two inequalities which in
our situation read
\begin{align*}
& \Re(w) > -\frac{(\Im(w))^{2}}{H^{2}}+3-\gamma^{2}\\
& 2\Re(w) > \frac{\Im(w)}{H^{2}}-3\gamma >0
\end{align*}
and the corresponding heart of the derived category is given by a
certain tilt $\operatorname{Coh}^{-\Im(w)}([S/\inv]) $ (see
\cite[\S~4.2]{Lim-Rota}).

Recalling that $\epsilon =\sum_{i}v_{i}$, we see that the slope
function associated to the central charge is exactly the Nironi slope
function $\mu_{\delta}$ with $\gamma =2\delta$. Consequently, in any
limit with $\Im(w)\to \infty$ (and satisfying the necessary
inequalities), the Lim-Rota stable objects become Nironi stable
sheaves. Thus we may identify $\M^{\dstable}_{\eta
(m\alpha_{d},v)}([S/\inv])$ with the moduli space of Lim-Rota stable
objects for $\gamma =2\delta$ and a choice of $w$ with appropriately
large $\Im(w)$.

We may now consider the derived Fourier-Mukai equivalence
\[
FM : D^{b}([S/\inv ])\to D^{b}(\Sresolve ). 
\]
This equivalence will take the Lim and Rota's stability conditions on
$[S/\inv ]$ to some Bridgeland stability condition on the $K3$ surface
$\Sresolve $. We claim these stability conditions are in fact the
stability conditions on $K3$ constructed by Bridgeland in
\cite{Bridgeland-K3-stability}. These stability conditions are
characterized by lying in the connected component of the stability
conditions on $\Sresolve$ where $\OO_{\pt}$ is stable. In
\cite[\S~5]{Lim-Rota}, Lim and Rota analyze the stability of
$\OO_{\pt}$ and show that for generic deformations of their stability
conditions, it is stable. Our claim follows.

We thus can make the identification
\[
\M^{\dstable}_{\eta (m\alpha_{d},v)}([S/\inv]) =
\M^{\mathsf{Bridgeland}}_{(0,m\gamma_{d}+v,1)}(\Sresolve ) 
\]
where the moduli space on the right is the moduli space of objects on
$\Sresolve$ with Mukai vector $(0,m\gamma_{d}+v,1)$ which are stable
with respect to the Bridgeland stability condition which is
Fourier-Mukai equivalent to our choosen Lim-Rota condition on $[S/\inv
]$.

We may now apply the results of Bayer and Macri \cite{Bayer-Macri},
who analyze in great detail all of the moduli spaces of Bridgeland
semistable objects on a $K3$ surface. They show that all moduli spaces of
objects semistable with respect to one of Bridgeland's constructed
stability conditions are deformation equivalent hyperkahler manifolds,
provided that the moduli space has no strictly semistable objects. In
particular, the moduli space
$\M^{\mathsf{Bridgeland}}_{(0,m\gamma_{d}+v,1)}(\Sresolve )$ is
deformation equivalent to $\M^{s}_{(0,m\gamma_{d}+v,1)}(\Sresolve )$
which proves the main assertion of the Proposition. Moreover,
Bayer-Macri analyze the Hilbert-Chow morphism which they show to be
a Lagrangian fibration (whenever it is well-defined) and compatible
with the deformation equivalence. 
\end{proof}

\section{Additional Examples}\label{sec: additional examples}

\subsection{Elliptically Fibered Calabi-Yau Threefold}

 Let $\pi : X \to B$ be an elliptically fibered Calabi-Yau threefold
with section $B\hookrightarrow X$ and with integral fibers.  We
additionally require the following conditions: \vskip1ex

\begin{enumerate}
\item There is an involution $\overline{\inv} : B \to B$ whose fixed
point locus is a smooth curve $C \subset B$.  \vskip1ex
\item $\pi |_{S} : S \to C$ is a smooth elliptic surface, where $S =
\pi^{-1}(C)$.  \vskip1ex
\item $\overline{\inv}$ lifts to a Calabi-Yau involution $\inv : X \to
X$ which restricts to a fiberwise action by $-1$ on $S$ over $C$.
\end{enumerate}
\vskip1ex

Let $[F] \in H_{2}(X)$ be the class of a fiber $F$ of $\pi : X \to
B$.  The following is our main result for this example:
\begin{theorem}\label{theorem: GV invariants Ell Fib CY3}
Conjecture \ref{conj: two tau-GV defns are the same nPT=nMT} holds for
the class $[F]$ and the invariants are given by
\[
n^{\PT}_{g,h}([F]) = n_{g,h}^{\MT}([F]) = 
\begin{cases}
-e(C) & (g,h)=(1,0), \\[5pt]
e(S) & (g,h) = (0,0), \\[5pt]
0 & \text{otherwise}
\end{cases}
\]
\end{theorem} 
\noindent We prove this theorem below by directly computing both the
$\inv$-PT and $\inv$-MT invariants in the class $[F]$.

The $\inv$-fixed locus $X^{\inv}$ is given by $S[2]$, the 2-torsion
points of the elliptic fibration $S\to C$. It admits a decomposition
\[
X^{\inv} = C_{0} \sqcup C_{1}
\]
where $C_{0}\cong C$ is the zero section, and $C_{1}$ is the locus of
non-trivial 2-torsion points. It is a degree 3 cover $C_{1}\to C$
which is simply ramified at the nodes of the nodal fibers and doubly
ramified at the cusps of the cuspidal fibers.

The compactified Jacobian $\Jac^{d}(X/B)$ of the family of elliptic
curves $X\to B$ can be identified with the moduli space of pure
dimension 1 stable sheaves $\calE$ with $ch(\calE )  = (0,0,[F],d)$.

The following facts are standard for this situation\footnote{See for
example \cite[bottom of page 2]{Migliorini-Shende2013} combined with
the fact that the moduli space of stable pairs supported on a family
of integral plane curves is isomorphic to the relative Hilbert scheme
of points \cite[top of page 3]{Migliorini-Shende2013} and
\cite[Prop.~1.8]{Pandharipande-Thomas1}.}
\begin{itemize}
\item $\Jac^{d}(X/B)\cong X$.
\item There is a rank $d$ bundle $V_{d}\to \Jac^{d}(X/B)$ whose fiber
over $\calE$ is $H^{0}(\calE )$. 
\item The map $\PT_{[F],d}(X)\to \Jac^{d}(X/B)$ given by $[\OO_{X}\xrightarrow{s} \calE ]\mapsto
\calE$ induces an isomorphism $\PT_{[F],d}(X)\cong \PP (V_{d})$.
\end{itemize}

The above constructions are $\inv$-equivariant and exhibit
$\PT_{[F],d}(X)$ as a projective bundle over $X$ so that we may
identify the fixed locus as follows: 
\[
\PT_{[F],d}(X)^{\inv} = \PP (V^{+}_{d})|_{X^{\inv}}\sqcup  \PP (V^{-}_{d})|_{X^{\inv}}
\]
where 
\[
V_{d}|_{X^{\inv}}\cong V^{+}_{d}|_{X^{\inv}} \oplus  V^{-}_{d}|_{X^{\inv}}
\]
is the eigenbundle decomposition for the $\inv$ action. Since
$X^{\inv}=C_{0}\sqcup C_{1}$, we see 
\[
\PT_{[F],d}(X)^{\inv} = \PP^{+}_{d,0}\sqcup  \PP^{+}_{d,1}\sqcup
 \PP^{-}_{d,0}\sqcup  \PP^{-}_{d,1}
\]
where we have abbreviated $\PP (V^{\pm}_{d})|_{C_{i}}$ as $\PP^{\pm}_{d,i}$.

Then since $\PP^{\pm}_{d,i}$ is a smooth projective bundle over a
curve we have
\[
\evir\left(\PP^{\pm}_{d,i} \right) = (-1)^{\dim V_{d}^{\pm}}\cdot \dim
V^{\pm}_{d}\cdot e\left(C_{i} \right).
\]
Using the decomposition given in Equation~\eqref{eqn: decomp of inv PT moduli space} :
\[
\PT_{[F],d}(X)^{\inv} = \bigsqcup_{d=2n+\epsilon} \PT_{[F],n,\epsilon}(X,\inv )
\]
we see that we need only compute $\dim V_{d}^{\pm}$ and $\epsilon$ for
each component. To do so, we pick $[s: \OO_{F} \to \calE ]\in
\PP^{\pm}_{d,i}$ where $\calE$ is an $\inv$-invariant sheaf supported
on a smooth fiber $F$ and $s\in H^{0}(\calE )^{\pm}$.  Let
$p_{0}=F\cap C_{0}$ and $\{p_{1},p_{2},p_{3} \}=F\cap C_{1}$ be the
origin and non-trivial two torsion points of $F$ respectively. Up to
renumbering, we may assume $\calE \mapsto p_{i}$ under the map
$\PP^{\pm}_{d,i}\to C_{i}$. The corresponding $\inv$-PT pair is
$[s:\OO_{F} \to \calE \otimes R_{\pm }]$ and so by definition
\begin{align*}
H^{0}(\calE \otimes R_{\pm }) & = nR_{\reg} +\epsilon R_{-}\\
&= nR_{+} +(n+\epsilon )R_{-}.
\end{align*}
On the other hand, 
\[
H^{0}(\calE \otimes R_{\pm}) = H^{0}(\calE )\otimes R_{\pm} =
(V^{+}_{d}\oplus V^{-}_{d})\otimes R_{\pm} 
\]
and so 
\[
\dim V_{d}^{\pm} = n = \half (d-\epsilon ).
\]
Moreover, since $s$ is an invariant section of $\calE \otimes
R_{\pm}$, it determines an isomorphism $\calE\otimes R_{\pm}\cong
\OO_{F}(D)$ where $D$ is an $\inv$-invariant effective divisor. Such
divisors are in the form given by Lemma~\ref{lem: chi(L) for L=O(D)
invariant under i} which by the same lemma shows 
\[
\epsilon =|T|-2
\]
where $T\subset \{0,1,2,3 \}$. Finally, the map $\PP^{\pm}_{d,i}\to
C_{i}$ which sends $\calE \mapsto p_{i}$ is given by summing the
points in the divisor $D$ in the group law of $F$. In particular we
must have
\[
\sum_{j\in T}p_{j} = p_{i}
\]
in the group law. For $p_{i}=p_{0}$ the only possible subsets are
\[
T = \emptyset, \{0 \},\{1,2,3 \},\{0,1,2,3 \}\text{ with $\epsilon
=-2,-1,1,2$ respectively,}
\]
and for $p_{i}=p_{1}$ the only possible subsets are
\[
T = \{1 \}, \{0,1 \},\{2,3 \},\{0,2,3 \}\text{ with $\epsilon
=-1,0,0,1$ respectively.}
\]
We then can compute the $Q^{[F]}$ coefficient of $Z^{\PT}(X,\inv )$ by
taking $\evir$ of the components of $\PT_{[F],d}(X)^{\inv}$,
multiplying by the appropriate $y^{n}w^{\epsilon}$, and then summing
over $n$ and $\epsilon$:

\begin{align*}
\left[Z^{\PT}(X,\inv ) \right]_{Q^{[F]}} & = \sum_{n=1}^{\infty}
(-1)^{n}n\,y^{n} \left((w^{-2}+w^{-1}+w+w^{2})e(C_{0}) +(w^{-1}+2+w^{1})e(C_{1}) \right)\\
&=\frac{-y}{(1+y)^{2}}\left((\psi_{w}^{2}-3\psi_{w})e(C_{0})
+\psi_{w}e(C_{1}) \right) \\
&= -e(C_{0}) \psi_{y}^{-1} \psi_{w}^{2} + (3
e(C_{0})-e(C_{1}))\psi_{y}^{-1} \psi_{w}. 
\end{align*}

Noting that $[\log Z^{\PT}(X,\inv )]_{Q^{[F]}} = [Z^{\PT}(X,\inv
)]_{Q^{[F]}}$ since $[F]$ is primitive, the formula for
$\nPT{g,h}{[F]}$ in Theorem~\ref{theorem: GV invariants Ell Fib CY3}
then follows from the fact that
\[
e(S)= 3e(C_{0}) - e(C_{1}).
\]
The above formula is easily proved by observing that the equality
holds when restricted to any fiber of $S\to C$ (smooth or singular). 

To compute the $\inv$-MT invariants, we use the following
\begin{lemma}\label{lem: MT moduli space = PT moduli space with n=1
for the class F}
$\M^{\epsilon}_{[F]}(X,\inv )\cong \PT_{[F],1,\epsilon}(X,\inv )$.
\end{lemma}
\begin{proof}
It suffices to show that for all $\calE \in \M^{\epsilon}_{[F]}(X,\inv
)$ there is a unique (up to scale) $\inv$-invariant section $s$. Then
the isomorphism in the lemma is given by 
\[
\calE \mapsto [s:\OO_{X}\to \calE ] .
\]
Any $\inv$-MT sheaf $\calE$ admits some $\inv$-invariant section $s$
since $\chi (\calE ) = R_{\reg}+\epsilon R_{-}$. We then get an
$\inv$-equivariant short exact sequence
\[
0 \to \OO_{F} \xrightarrow{s} \calE \to Q \to 0
\]
where $F=\supp (\calE )$ is some fiber of $S\to C$ and $Q$ is
necessarily 0 dimensional since $F$ is integral. Moreover, $\chi
(\OO_{F})=R_{+}-R_{-}$ for any fiber $F$ and so $\chi (Q) = H^{0}(Q) =
(2+\epsilon )R_{-}$. Taking the long exact sequence associated to the
above short exact sequence, and then restricting to the
$\inv$-invariant part yields an isomorphism
\[
H^{0}(\OO_{F})\cong H^{0}(\calE )^{\inv}
\]
so that $H^{0}(\calE )^{\inv}$ is 1 dimensional and hence generated by $s$.
\end{proof}

The lemma then allows us to apply our previous analysis of the
components of $\PT_{[F],d,\epsilon}(X,\inv )$ specialized to the case
$d=1$ where $\PP^{\pm}_{1,i}\cong C_{i}$. As before, the components
and the corresponding $\epsilon$ is determined by the subset $T\subset
\{0,1,2,3 \}$. The result is
\[
\M^{\epsilon}_{[F]}(X,\inv )\cong \begin{cases}
C_{0} & \text{if $\epsilon =\pm 2$,}\\
C_{0}\sqcup C_{1} & \text{if $\epsilon =\pm 1$,}\\
C_{1}\sqcup C_{1} & \text{if $\epsilon =0$.}\\
\end{cases}
\]
Since $\M^{\epsilon}_{[F]}(X,\inv )$ is a smooth curve, the perverse
sheaf of vanishing cycles is the shifted constant sheaf,
$\phi^{\bullet} = \QQ [1].$ Furthermore, $\Chow_{[F]}(X)^{\inv}$ here
is $C$ with $\pi^{\epsilon}: \M^{\epsilon}_{[F]}(X,\inv )\to C$ given
by projection $C_{i}\to C$ on each component.

In general, if 
\[
\pi :C''\to C'
\]
is a proper surjective morphism of smooth curves, then $\pi$ is
semi-small and consequently $R^{\bullet}\pi_{*}\QQ [1]$ is a perverse
sheaf on $C'$  \cite[Theorem~4.2.7]{deCataldo-Migliorini2009} and so
\[
^{p}H^{i}(R^{\bullet}\pi_{*}\QQ [1]) = \begin{cases}
R^{\bullet}\pi_{*}\QQ [1] & i=0,\\
0 & i\neq 0.
\end{cases}
\]
Consequently we have
\begin{align*}
\sum_{i} \chi (^{p}H^{i}(R^{\bullet} \pi_{*}\QQ [1]))y^{i}&= \chi (R^{\bullet} \pi_{*}\QQ [1])\\
&=-e(C'') 
\end{align*}
where the last equality follows from the perverse Leray spectral sequence.

Applying this to $\pi^{\epsilon}:\M_{[F]}^{\epsilon}(X,\inv )\to C$ we
can then compute the Maulik-Toda polynomial:
\begin{align*}
\MT_{[F]}(y,w) &= -e(C_{0}) (w^{-2}+w^{-1}+w^{1}+w^{2}) -e(C_{1})(w^{-1}+2+w^{1})\\
&=-e(C) \psi_{w}^{2} +e(S)\psi_{w}
\end{align*}
and the formula for $\nMT{g,h}{[F]}$ follows. The proof of
Theorem~\ref{theorem: GV invariants Ell Fib CY3} is then complete.

\subsection{The Local Football}

Let $\inv$ be an involution of $C\cong \PP^{1}$ fixing two points
$z_{0}$ and $z_{\infty }$. The line bundles $\OO_{C}(-z_{0})$ and
$\OO_{C}(-z_{\infty})$ are naturally $\inv$-equivariant and
consequently the CY3
\[
X = \Tot (\OO_{C}(-z_{0})\oplus \OO_{C}(-z_{\infty}))
\]
has a natural involution which we also call $\inv$.

The global stack quotient $[X/\inv]$ is called a \emph{local
football}, though here we use this term to mean the pair $(X, \inv)$.
The purpose of this section is to give a proof of
Proposition~\ref{prop: nPT=nMT=1 for (g,h,d)=1 for local football}
which we restate here:
\begin{proposition}
For all $d>0$, we have
\[
n^{\PT}_{g,h}(d[C]) = n^{\MT}_{g,h}(d[C])=
\begin{cases}
1 & (d,g,h) = (1,0,0) \\
0 & \text{otherwise}
\end{cases}
\]
\end{proposition}  

We start with $\inv$-PT theory. As we did in Lemma~\ref{lem:
ZPT(X,i)=ZPT([X/i])|wi=w}, we may compute the $\inv$-PT invariants of
$(X,\inv)$ by computing orbifold PT invariants of the stack quotient
$[X/\inv ]$. Since $[X/\inv ]$ is toric, we may use the orbifold
topological vertex \cite{Bryan-Cadman-Young}. The partition function
$Z^{\PT}([X/\inv ])$ is computed in Section~4.2 of
\cite{Bryan-Cadman-Young} where it is given by
$DT'(\mathcal{X}_{2,2})$ in the notation of
\cite{Bryan-Cadman-Young}. Here we are using
\cite[Theorem~6.12]{Beentjes-Calabrese-Rennemo} which states that the
reduced Donaldson-Thomas partition function is equal to the
Pandharipande-Thomas partition function. By the formula just below
\cite[Prop.~3]{Bryan-Cadman-Young} we have
\begin{equation}\label{eqn: ZPT of local football from BCY}
Z^{\PT}([X/\inv ]) = \prod_{u\in \{v,vp_{0},vr_{0},vp_{0}r_{0} \}} M(u,-q)^{-1}
\end{equation}
where 
\[
M(x,q) = \prod_{n=1}^{\infty}(1-xq^{n})^{-n}.
\]
The variables $v, p_{0},r_{0},$ and $q$ track classes in the
$K$-theory of sheaves on $[X/\inv ]$ which we can equivalently regard
as $\inv$-equivariant sheaves on $X$. Recall that the variables in
$Z^{\PT}(X,\inv )$ track the curve class and $\chi$ of the sheaf. Thus
to specialize $Z^{\PT}([X/\inv ])$ to $Z^{\PT}(X,\inv )$, we must
compute the curve class and $\chi$ of each of these classes. The
results are given in the following table:

\vskip3ex

\setlength{\arrayrulewidth}{0.8mm}
\setlength{\tabcolsep}{13pt}
\renewcommand{\arraystretch}{1.5}
\hskip-12ex
\small
\begin{table}[H]
\centering
\begin{tabular}{lllll}
\toprule
$Z^{\PT}([X/\inv ])$  & Class in $[X/\inv ]$ & Equivariant class & $(\chi ,\beta )$ & $Z^{\PT}(X,\inv )$ \\
variable&  (see \cite[\S~3.3]{Bryan-Cadman-Young}) & on $X$ &of class & variable\\
\midrule
$p_{0}$ & $[\OO_{z_{0}}\otimes R_{-}]$ & $[\OO_{z_{0}}\otimes R_{-}]$ &
$(R_{-},0)$ &$w$\\
$r_{0}$ & $[\OO_{z_{\infty} }\otimes R_{-}]$ & $[\OO_{z_{\infty} }\otimes R_{-}]$ &
$(R_{-},0)$ &$w$\\
$q$ & $[\OO_{\overline{p} }]$ & $[\OO_{p }\oplus \OO_{\inv (p)}]$ &
$(R_{\reg },0)$ &$y$\\
$v$ & $[\OO_{[C/\inv ]}(-\overline{p})]$&$[\OO_{C }(-p-\inv (p))]$ &
$(-R_{- },[C])$ &$Qw^{-1}$\\
\bottomrule
\end{tabular}
\end{table}
\normalsize

\smallskip

In the above table,  $\,\,\overline{p}\in [C/\inv ]$ is a generic point
corresponding to the $\inv$-orbit $\{p,\inv (p) \}\subset C$ and the
formula $\chi (\OO_{C}(-p-\inv (p)))=-R_{-}$ is obtained by applying
Lemma~\ref{lem: chi(L) for L=O(D) invariant under i}.

Equation~\eqref{eqn: ZPT of local football from BCY} then specializes
to 
\[
Z^{\PT}(X,\inv ) = M(Qw^{-1},-y)^{-1}  M(Q,-y)^{-2}  M(Qw,-y)^{-1}.
\]

It is straightforward to show that
\[
\log M(x,y)^{-1} = \sum_{k=1}^{\infty } \frac{x^{k}}{k}\, \psi_{-(-y)^{k}}^{-1}.
\]
and so
\begin{align*}
\log Z^{\PT}(X, \inv) &= \sum_{k=1}^{\infty } \frac{1}{k}Q^{k}\cdot
\left(w^{-k}+2+w^{k} \right)
\cdot  \psi_{-(-y)^{k}}^{-1} \\
&= \sum_{k=1}^{\infty } \frac{1}{k}Q^{k} \cdot
\psi_{-(-y)^{k}}^{-1}\cdot \psi_{w^{k}}
\end{align*}
and the formula for $\nPT{g,h}{d[C]}$ then follows.

Turning now to the $\inv$-MT theory, we begin with the following key result.
\begin{lemma}
$\M^{\epsilon}_{d[C]}(X,\inv )$ is empty for $d>1$.
\end{lemma}

\begin{proof}
Since $\M^{\epsilon}_{d[C]}(X,\inv )$ is proper over $\Chow_{d[C]}(X)$
which is a point, if it is non-empty, it admits a fixed point for the
torus action induced from the action on $X$. In order to obtain a
contradiction, we suppose there exists a torus invariant $\inv$-MT
sheaf $\calE $ in the class $d[C]$ with $d>1$.

Then the scheme-theoretic support of $\calE$ is the thickened curve
$C_{\lambda}$ determined from a $2$-dimensional partition $\lambda$.
Since
$\chi(\calE) = R_{\reg}+\epsilon R_{-}$, there is a 
$\inv$-invariant section
\[
0 \to \OO_{C_{\lambda}} \xrightarrow{s} \calE \to Q \to 0,
\]
and then since $\calE$ is $\inv$-stable, 
$\chi(\OO_{C_{\lambda}}) = kR_{\reg}+ m R_{-}$ with $k \leq 1$.

Let $\pi : X \to C$ be the projection, then $\chi(\OO_{C_{\lambda}}) =
\chi(\pi_{*}\OO_{C_{\lambda}})$.  We have
\[
\pi_{*} \OO_{C_{\lambda}} = \bigoplus_{(i,j) \in \lambda}
\OO_{C}(z_{0})^{i} \otimes \OO_{C}(z_{\infty})^{j} =  \bigoplus_{(i,j)
\in \lambda}\OO_{C}(iz_{0}+jz_{\infty }) .
\]
Writing $i=2a+i'$ and $j=2b+j'$ with $a,b\geq 0$ and $(i',j')\in
\{(0,0),(0,1),(1,0),(1,1) \}$ and applying Lemma~\ref{lem: chi(L) for
L=O(D) invariant under i}, we get 
\[
\chi (\OO_{C}(iz_{0}+jz_{\infty}) = (1+a+b)R_{\reg} + (i'+j'-1)R_{-}. 
\]
Since we are taking sums of these terms, the only way to guarantee $k
\leq 1$ as above, is if $\lambda = \square$ is the unique partition of
length $1$.

Thus $\calE$ is scheme-theoretically supported on $C$, and hence it
must a torus invariant, $\inv$-equivariant, rank $d$ vector bundle on
$C$. Such vector bundles split as a direct sum of $\inv$-equivariant
lines bundles which then contradicts the $\inv$-stability of $\calE$.
\end{proof}

It follows from the lemma that $\nMT{g,h}{d[C]}=0$ for $d>1$ and for
$d=1$ we apply Corollary~\ref{cor: ngh(C)=1 for local curve} which
then finishes the proof of Proposition~\ref{prop: nPT=nMT=1 for
(g,h,d)=1 for local football}.

\section{Acknowledgments}
The authors warmly acknowledge and thank the following people for
helpful discussions: Arend Bayer, John Duncan, Bronson Lim, Davesh
Maulik, Georg Oberdieck, Franco Rota, and Junliang Shen.

\vfill

\pagebreak

\appendix

\section{Tables of Values for $n_{g,h}(d,\oddAb )$}
\renewcommand{\arraystretch}{1.0}

In this appendix, we list explicitly the values of $n_{g,h}(d,\oddAb
)$ for $d\leq 7$. This case includes the primitive class $\beta_{d}$
on a Picard rank one Abelian surface $S$ where these numbers have some
enumerative significance. The $h=0$ numbers are (up to the minus sign
due to the second factor in $X=S\times \CC$) actual counts of
$\inv$-invariant hyperelliptic curves on $S$; they coincide with the
counts computed in \cite{BOPY}. For each $d$, the highest genus
occuring is $g=d+1$, the arithmetic genus of the class $\beta_{d}$.  Let
\[
\Chow_{\beta_{d}}(X)_{h}\subset \Chow_{\beta_{d}}(X)^{\inv }
\]
be the
dimension $h$ components of the $\inv$-fixed point locus. Then one can
show that
\[
n_{d+1,h}(d,\oddAb ) = \evir \left(\Chow_{\beta_{d}}(X)_{h} \right). 
\]
The right hand side can be computed directly since
$\Chow_{\beta_{d}}(X)^{\inv} = \Chow_{\beta_{d}}(S)^{\inv} \times \CC$
and that the first factor here is the disjoint union of the
$\inv$-invariant linear systems of the $\inv$-invariant line bundles
in the class $\beta_{d}$.

For $d\leq 1$, the only non-zero values of $n_{g,h}(d,\oddAb )$ are given
by
\[
n_{1,0}(0,\oddAb )=-4,\quad n_{2,0}(1,\oddAb )=-16.
\]
For $2\leq d\leq 7$, see the table below:

\newcommand{\ra}[1]{\renewcommand{\arraystretch}{#1}}

\tiny
\setlength{\tabcolsep}{7pt}
\begin{table}[H]
\centering
\caption*{Non-zero values of $n_{g,h}(d,\oddAb )$ for $2\leq d\leq 7$. }
\ra{1.2}
\begin{tabular}{@{}lllllcllll@{}}\toprule
& \multicolumn{2}{c}{$\bf{d=2}$}&& & & \multicolumn{2}{c}{$\bf{d = 3}$}& &
\\
\cmidrule{2-3} \cmidrule{7-8}
& $h=0$ & $h=1$ & & && $h=0$ & $h=1$ & & \\ \midrule
$g=2$ & -48 &  & & && -64 &  & &  \\
$g=3$ & -24 & 8& & && -160 & & & \\
$g=4$ &     &  & & && -16 & 32& & \\ \bottomrule

\addlinespace[2em]

& \multicolumn{2}{c}{$\bf{d=4}$}&& & & \multicolumn{2}{c}{$\bf{d = 5}$}& & \\
\cmidrule{2-4} \cmidrule{7-9}
& $h=0$ & $h=1$ & $h=2$ &&& $h=0$ & $h=1$ & $h=2$& \\ \midrule
$g=2$ & -112 &  & & && -96 &  & &  \\
$g=3$ & -456 & 24& & && -1056 &  && \\
$g=4$ & -192 & 96& & && -912 & 224& & \\
$g=5$ & -4   & 48&-12&&& -96 & 320 && \\
$g=6$ &    &   &  &&&      & 32  &-48& \\ \bottomrule

\addlinespace[2em]
& \multicolumn{2}{c}{$\bf{d=6}$} &&&& \multicolumn{2}{c}{$\bf{d=7}$}  & &  \\
\cmidrule{2-5} \cmidrule{7-10}
& $h=0$ & $h=1$ & $h=2$ & $h=3$& & $h=0$  &  $h=1$  &  $h=2$ &  $h=3$ \\ \midrule
$g=2$ & -192 &&&&&  -128  & && \\
$g=3$ & -1920 &32 &&&&  -3264    &&& \\
$g=4$ & -2992 &	  512 &	&	&	&-7776	&704	&&\\
$g=5$ & -736 &	  1056 &	-64 &	&       &-3424	&3072	&	&\\
$g=6$ & -16 &	  384 &		-144 &	&       &-240	&1920	&-448	&\\
$g=7$ &    &  8 & -72 &		16 & &	                &192	&-480&\\ 
$g=8$ &    &    &  &		 & &	&	&-48& 64\\ \bottomrule

\end{tabular}
\end{table}

\normalsize


\bibliography{localbiblio}

\begin{thebibliography}{10}

\bibitem{Compact-Complex-Surfaces-BHPV}
Wolf~P. Barth, Klaus Hulek, Chris A.~M. Peters, and Antonius Van~de Ven.
\newblock {\em Compact complex surfaces}, volume~4 of {\em Ergebnisse der
  Mathematik und ihrer Grenzgebiete. 3. Folge. A Series of Modern Surveys in
  Mathematics [Results in Mathematics and Related Areas. 3rd Series. A Series
  of Modern Surveys in Mathematics]}.
\newblock Springer-Verlag, Berlin, second edition, 2004.

\bibitem{Bayer-Macri}
Arend Bayer and Emanuele Macr\`\i.
\newblock M{MP} for moduli of sheaves on {K}3s via wall-crossing: nef and
  movable cones, {L}agrangian fibrations.
\newblock {\em Invent. Math.}, 198(3):505--590, 2014.

\bibitem{Beentjes-Calabrese-Rennemo}
Sjoerd~Viktor Beentjes, John Calabrese, and J{\o}rgen~Vold Rennemo.
\newblock {A proof of the Donaldson-Thomas crepant resolution conjecture}.
\newblock arXiv:math/1810.06581.

\bibitem{Behrend-micro}
Kai Behrend.
\newblock Donaldson-{T}homas type invariants via microlocal geometry.
\newblock {\em Ann. of Math. (2)}, 170(3):1307--1338, 2009.
\newblock arXiv:math/0507523.

\bibitem{BBDJS}
C.~Brav, V.~Bussi, D.~Dupont, D.~Joyce, and B.~Szendr\H{o}i.
\newblock Symmetries and stabilization for sheaves of vanishing cycles.
\newblock {\em J. Singul.}, 11:85--151, 2015.
\newblock With an appendix by J\"{o}rg Sch\"{u}rmann.

\bibitem{Bridgeland-K3-stability}
Tom Bridgeland.
\newblock Stability conditions on {$K3$} surfaces.
\newblock {\em Duke Math. J.}, 141(2):241--291, 2008.

\bibitem{Bridgeland-PTDT}
Tom Bridgeland.
\newblock Hall algebras and curve-counting invariants.
\newblock {\em J. Amer. Math. Soc.}, 24(4):969--998, 2011.
\newblock arXiv:1002.4374.

\bibitem{Bryan-Cadman-Young}
Jim Bryan, Charles Cadman, and Ben Young.
\newblock The orbifold topological vertex.
\newblock {\em Adv. Math.}, 229(1):531--595, 2012.
\newblock arXiv:math/1008.4205.

\bibitem{BOPY}
Jim Bryan, Georg Oberdieck, Rahul Pandharipande, and Qizheng Yin.
\newblock Curve counting on abelian surfaces and threefolds.
\newblock {\em Algebr. Geom.}, 5(4):398--463, 2018.
\newblock arXiv:math/1506.00841.

\bibitem{deCataldo-Migliorini2009}
Mark de~Cataldo and Luca Migliorini.
\newblock The decomposition theorem, perverse sheaves and the topology of
  algebraic maps.
\newblock {\em Bulletin of the American Mathematical Society}, 46(4):535--633,
  2009.

\bibitem{Garbagnati-Sarti-even-set-2008}
Alice Garbagnati and Alessandra Sarti.
\newblock Projective models of {$K3$} surfaces with an even set.
\newblock {\em Adv. Geom.}, 8(3):413--440, 2008.

\bibitem{Garbagnati-Sarti-Kummer-surfaces-2016}
Alice Garbagnati and Alessandra Sarti.
\newblock Kummer surfaces and {K}3 surfaces with {$(\Bbb{Z}/2\Bbb{Z})^4$}
  symplectic action.
\newblock {\em Rocky Mountain J. Math.}, 46(4):1141--1205, 2016.

\bibitem{Go-Va}
Rajesh Gopakumar and Cumrun Vafa.
\newblock M-theory and topological strings--{II}, 1998.
\newblock Preprint, hep-th/9812127.

\bibitem{KKV}
Sheldon Katz, Albrecht Klemm, and Cumrun Vafa.
\newblock M-theory, topological strings and spinning black holes.
\newblock {\em Adv. Theor. Math. Phys.}, 3(5):1445--1537, 1999.
\newblock Preprint: hep-th/9910181.

\bibitem{Lim-Rota}
Bronson Lim and Franco Rota.
\newblock Characteristic classes and stability conditions for projective
  {K}leinian orbisurfaces.
\newblock {\em Math. Z.}, 300(1):827--849, 2022.
\newblock arXiv:math.AG/2001.09139.

\bibitem{Macdonald-poincare-poly}
I.~G. Macdonald.
\newblock The {P}oincar\'e polynomial of a symmetric product.
\newblock {\em Proc. Cambridge Philos. Soc.}, 58:563--568, 1962.

\bibitem{Maulik-Toda}
Davesh Maulik and Yukinobu Toda.
\newblock Gopakumar-{V}afa invariants via vanishing cycles.
\newblock {\em Invent. Math.}, 213(3):1017--1097, 2018.

\bibitem{Migliorini-Shende2013}
Luca Migliorini and Vivek Shende.
\newblock A support theorem for {H}ilbert schemes of planar curves.
\newblock {\em Journal of the European Mathematical Society}, 15(6):2353--2367,
  2013.

\bibitem{Morrison-K3surfaces-1984}
D.~R. Morrison.
\newblock On {$K3$} surfaces with large {P}icard number.
\newblock {\em Invent. Math.}, 75(1):105--121, 1984.

\bibitem{Nironi-stability-on-stacks}
Fabio Nironi.
\newblock {Moduli Spaces of Semistable Sheaves on Projective Deligne–Mumford
  Stacks}.
\newblock arXiv:math/0811.1949.

\bibitem{Pandharipande-Thomas-KKV}
R.~Pandharipande and R.~P. Thomas.
\newblock The {K}atz--{K}lemm--{V}afa conjecture for {$K3$} surfaces.
\newblock {\em Forum Math. Pi}, 4:e4, 111, 2016.
\newblock arXiv:math/1404.6698.

\bibitem{Pandharipande-Thomas1}
Rahul Pandharipande and Richard Thomas.
\newblock Curve counting via stable pairs in the derived category.
\newblock {\em Invent. Math.}, 178(2):407--447, 2009.
\newblock arXiv:math/0707.2348.

\bibitem{Pandharipande-Thomas3}
Rahul Pandharipande and Richard Thomas.
\newblock Stable pairs and {BPS} invariants.
\newblock {\em J. Amer. Math. Soc.}, 23(1):267--297, 2010.
\newblock arXiv:math/0711.3899.

\bibitem{Pietromonaco-GHilbA}
Stephen Pietromonaco.
\newblock {$G$}-invariant {H}ilbert schemes on {A}belian surfaces and
  enumerative geometry of the orbifold {K}ummer surface.
\newblock {\em Res. Math. Sci.}, 9(1):Paper No. 1, 21, 2022.
\newblock arXiv:math/2011.14020.

\bibitem{Popa2017}
Mihnea Popa.
\newblock Derived equivalences of smooth stacks and orbifold {H}odge numbers.
\newblock In {\em Higher dimensional algebraic geometry---in honour of
  {P}rofessor {Y}ujiro {K}awamata's sixtieth birthday}, volume~74 of {\em Adv.
  Stud. Pure Math.}, pages 357--380. Math. Soc. Japan, Tokyo, 2017.

\bibitem{Rose-2014}
Simon C.~F. Rose.
\newblock Counting hyperelliptic curves on an {A}belian surface with
  quasi-modular forms.
\newblock {\em Commun. Number Theory Phys.}, 8(2):243--293, 2014.

\bibitem{Shen-Yin-hyperkahler}
Junliang Shen and Qizheng Yin.
\newblock {Topology of Lagrangian fibrations and Hodge theory of hyper-K\"ahler
  manifolds}.
\newblock arXiv:math/1812.10673.

\bibitem{Toda-PTDT}
Yukinobu Toda.
\newblock Curve counting theories via stable objects {I}. {DT}/{PT}
  correspondence.
\newblock {\em J. Amer. Math. Soc.}, 23(4):1119--1157, 2010.

\end{thebibliography}
\bibliographystyle{plain}

\end{document}